\documentclass[12pt,reqno]{article}

\usepackage[margin=1in]{geometry}
\usepackage{amsthm,amsmath,epsfig,amsfonts,amssymb,bm,xcolor,amsbsy,times,dsfont,amscd,latexsym,amsxtra}
\usepackage{bbm}
\usepackage{enumerate}
\usepackage{mathrsfs}
\usepackage{graphicx}
\usepackage{subfig}
\usepackage{comment}
\usepackage{mathtools}
\usepackage[title]{appendix}
\usepackage{cases}
\usepackage{ifpdf}
\usepackage{color}
\definecolor{webgreen}{rgb}{0,.5,0}
\definecolor{webbrown}{rgb}{.8,0,0}
\definecolor{emphcolor}{rgb}{0.95,0.95,0.95}

\usepackage{hyperref}
\hypersetup{hidelinks,
          colorlinks=true,
          linkcolor=blue,
          filecolor=blue,
          citecolor=blue,
          urlcolor=blue,
          breaklinks=true}

\newtheorem{theorem}{Theorem}[section]
\newtheorem{assumption}[theorem]{Assumption}

\newtheorem{definition}[theorem]{Definition}
\newtheorem{example}[theorem]{Example}

\newtheorem{lemma}[theorem]{Lemma}

\newtheorem{remark}[theorem]{Remark}
\numberwithin{equation}{section}

\newcommand{\bu}{\mathbf{u}}
\newcommand{\X}{\mathbb{X}}
\newcommand{\U}{\mathcal{U}}
\newcommand{\D}{\mathcal{D}}
\newcommand{\A}{\mathcal{A}}
\newcommand{\ba}{\mathbf{a}}
\newcommand{\tr}{\mathrm{tr}}

\title{Equilibrium stochastic control with implicitly defined objective functions}
\author{Zongxia Liang\thanks{ Email: liangzongxia@mail.tsinghua.edu.cn, 
		Department of Mathematical Sciences, Tsinghua University, Beijing 100084, China.}
	\and Jianming Xia\thanks{ Email: xia@amss.ac.cn, 
	State Key Laboratory of Mathematical Sciences, Academy of Mathematics and Systems Science, Chinese Academy of Sciences, Beijing 100190, China.}
	\and Keyu Zhang\thanks{ Email: zhangky21@mails.tsinghua.edu.cn, 
		Department of Mathematical Sciences, Tsinghua University, Beijing 100084, China.}}

\date{Dedicated to Professor Shige Peng on the Occasion of His 80th Birthday}

\begin{document}

\maketitle
\begin{abstract}
	This paper considers a class of stochastic control problems with implicitly defined objective functions, which are the sources of  time-inconsistency. We study the closed-loop equilibrium  in a general controlled diffusion framework. First, we provide a sufficient and necessary condition for a closed-loop strategy to be an equilibrium. Then, we apply the result to discuss two problems of dynamic portfolio selection for a class of betweenness preferences, allowing for closed convex constraints on portfolio weights and borrowing cost, respectively. The equilibrium portfolio strategies are explicitly characterized in terms of the solutions of some first-order ordinary differential equations for the case of deterministic market coefficients. 
	
	\noindent \textbf{Mathematics Subject Classification (2020)}: 91G10, 93E20
	\\
	\noindent {\small\textbf{Keywords:}} stochastic control, implicitly defined objective function, equilibrium strategy, portfolio selection, betweenness preference
\end{abstract}

\section{Introduction}
This paper is concerned with a class of  stochastic control problems, in which the objective function is implicitly defined; see (\ref{F}) in Section \ref{sec:pro:formulation}. In general, this kind of problem inherently exhibit time inconsistency: a decision deemed optimal today may no longer be optimal at a certain future time. In mathematical terms, the classic dynamic programming principle, namely, the Bellman optimality principle, is not satisfied. We follow the intra-personal equilibrium approach to treat the time inconsistency,  
by which the agent makes a decision taking  into account the future selves' behavior and looks for the best current action in response to that to ensure the strategy's time-consistency.
 
The intra-personal equilibrium approach dates back to \cite{Strotz1955}. A precise mathematical formalization of such an equilibrium in continuous time was first given by \cite{Ekeland2006} for a deterministic Ramsey model with non-exponential discounting. It is then developed, among others, by
\cite{Ekeland2008}, 
\cite{Bjork-2017}, \cite{Hu-2012}, \cite{Yong2012} and \cite{Hejiang2021} in continuous-time stochastic control. 

As far as we know, in almost all of the existing literature on continuous-time stochastic control, the objective functions are explicitly defined.\footnote{The only exception is the recent work \cite{liang2023dynamic}, which investigates the open-loop equilibrium {solutions} to the continuous-time portfolio selection for betweenness preferences.} On the other hand, the implicitly defined objective functions have broad interests in economics and finance, including as important special cases the betweenness preference functionals  
of \cite{Chew-1983,Chew-1989,Dekel-1986}, the risk indices of \cite{Aumann2008,Foster2009}, and the expectile-type risk measures of \cite{Newey1978,Kuan2009,Bellini2014,MaoCai2018}.  
The aim of the present paper is to study the closed-loop equilibrium {solutions} for the continuous-time stochastic control with implicitly defined objective functions. 
Our first contribution is to provide a sufficient and necessary condition for a strategy to be an equilibrium strategy.  
Our second contribution is to apply the result to study the continuous-time portfolio selection problems with betweenness preferences.\footnote{Portfolio selection problems are well-studied with the expected utility (EU) theory; see,  for instance, \cite{Merton1971}. EU has the advantage of time-consistency so that dynamic programming principle can be applied. 
However, various empirical and experimental studies contradicts EU theory, as exemplified by the famous Allais paradox of \cite{Allais-1953}. This motivates the emergence of many theories on alternative preferences, including the betweenness preferences; see, for instance, \cite{Chew-1983,Chew-1989} and \cite{Dekel-1986}.}
Specifically, we consider the dynamic portfolio selection with constant-relative-risk-aversion (CRRA) betweenness preferences, with closed convex constraints on portfolio weights or with borrowing cost: the interest rate for borrowing money being larger than for saving.\footnote{Constrained portfolio optimization problems with EU are well-documented in the literature, see, for instance, \cite{karatzas-1992}, \cite{Hu-2005} and \cite{Cuoco-1997}.   The works related to the borrowing cost with EU and mean-variance criterion include \cite{Flem-1991, Xu-1998, Guan2020, Guan2023} and \cite{Yang2019}, among others.}   When the market coefficients are deterministic, we explicitly characterize the closed-loop equilibrium strategies in terms of solutions of some first-order ordinary differential equations (ODEs) and show the uniqueness within continuous state-independent strategies. In particular, when the constraint is removed, we recover the solution in \cite{liang2023dynamic}.  Moreover, it is conceivable that the problems with constant-absolute-risk aversion (CARA) betweenness preference (see, e.g. \cite{liang2023dynamic}) are tractable as well.

   
The rest of the paper is organized as follows. Section \ref{sec:pro:formulation} introduces a
general diffusion framework of stochastic control problems with implicitly defined objective functions and formalizes the definition of closed-loop equilibrium strategies. Section \ref{sec:suffandness} gives a sufficient and necessary condition for a strategy to be an equilibrium strategy. In Section \ref{sec:example}, we apply the result in Section \ref{sec:suffandness} to study portfolio selection problems. 

\section{The model}\label{sec:pro:formulation}
\subsection{Notation}
We introduce some notations first. Throughout the paper, the state space $\X$ is either $(0,+\infty)$ or $\mathbb{R}^{n}$, the $n$-dimensional Euclidean space. Denote the Euclidean norm of a vector $x$ as $\lVert x\rVert$. Denote by
$A^{\top}$ the transpose of matrix $A$, and  by $\tr(A)$ the trace of a square matrix $A$. 

Let $\mathbb{T}$ be a certain set. Consider a map $\xi:\mathbb{T}\times\X\ni(t,x) \mapsto \xi(t,x)\in\mathbb{R}^{l}$. We say that $\xi$ is locally Lipschitz in $x\in\X$, uniformly in $t\in\mathbb{T}$, if there exists a sequence $\left\{\X_{k}\right\}_{k\geq1}$ of compact sets with $\cup_{k\geq1}\X_{k}=\X$ and a sequence $\left\{L_{k}\right\}_{k\geq1}$ of positive numbers such that, for any $k\geq1$, $\lVert\xi(t,x)-\xi(t,x')\rVert\leq L_{k}\lVert x-x'\rVert, \, \forall y\in\mathbb{T},\, x,x' \in\X_{k}.$ In the case $\mathbb{X}=\mathbb{R}^{n}$, we say that $\xi$ is of linear growth in $x\in\mathbb{X}$, uniformly in $t\in\mathbb{T}$, if there exists $L>0$ such that $\lVert\xi(t,x)\rVert\leq L\left(1+\lVert x\rVert\right), \forall (t,x)\in\mathbb{T}\times\X$. In the case $\mathbb{X}=\left(0,+\infty\right)$, we say that $\xi$ has bounded norm in $x\in\mathbb{X}$, uniformly in $t\in\mathbb{T}$, if there exists $L>0$ such that $\lVert\xi(t,x)\rVert\leq Lx, \forall (t,x)\in\mathbb{T}\times\X$. 

Let $\mathbb{Z}$ be a topology space. We say that a function $L:\mathbb{Z}\to(0,+\infty)$ is locally bounded if $L$ is bounded on every compact subset of $\mathbb{Z}$. 
We say that a map $\xi:\X\times\mathbb{Z}\ni(x,z)\mapsto\xi(x,z)\in\mathbb{R}^l$ is of polynomial growth in $x\in\X$, local-uniformly in $z\in\mathbb{Z}$, if there exist a locally bounded function $L: \mathbb{Z}\to(0,+\infty)$ and an integer $\gamma\geq1$ such that $\lVert \xi(x,z)\rVert\leq L(z)\left(1+\varphi_{2\gamma}(x)\right),\, (x,z)\in\mathbb{X}\times\mathbb{Z}$,
where 
		\begin{equation}\label{poly}
		\varphi_{2\gamma}(x)=
		    \begin{cases}
		        \lVert x\rVert^{2\gamma},\,&\text{if}\quad \mathbb{X}=\mathbb{R}^n,\\
		        x^{2\gamma}+ x^{-2\gamma},\,&\text{if}\quad \mathbb{X}=(0,+\infty).
		    \end{cases}
		  \end{equation}
We say a map 
$\xi: \mathbb{T}\times\X\times\mathbb{Z}\ni(t,x,z)\mapsto \xi(t,x,z)\in\mathbb{R}^{l}$ is of polynomial growth in $x\in\X$, local-uniformly in $z\in\mathbb{Z}$ and uniformly in $t\in\mathbb{T}$, if there exist a locally bounded function $L: \mathbb{Z}\to(0,+\infty)$ and an integer $\gamma\geq1$ such that $\lVert \xi(t,x,z)\rVert\leq L(z)\left(1+\varphi_{2\gamma}\left(x\right)\right), \forall (t,x,z)\in\mathbb{T}\times\X\times\mathbb{Z}$, where $\varphi_{2\gamma}$ is defined by (\ref{poly}).

\subsection{Problem Formulation}
Let $\left(\Omega,\mathcal{F},\mathbb{F},\mathbb{P}\right)$ be a filtered probability space, where $\mathbb{F}=\left\{\mathcal{F}_{t}\right\}_{0\leq t\leq T}$ is the filtration generated by a standard $d$-dimensional Brownian motion, $W(t):= \bigl(W_{1}(t),\cdots,W_{d}(t)\bigr)^{\top}$, $0\leq t\leq T$, augmented by all null sets. Consider the following controlled stochastic differential equation (SDE):
\begin{equation}\label{problem}
	\begin{cases}
		 dX^{\textbf{u}}(s)=\mu(s,X^{\textbf{u}}(s),\textbf{u}(s,X^{\textbf{u}}(s)))ds+\sigma(s,X^{\textbf{u}}(s),\textbf{u}(s,X^{\textbf{u}}(s)))dW(s),\, s\in[t,T],\, \\X^{\textbf{u}}(t)=x,
	\end{cases}
\end{equation}
where 
\begin{itemize}
\item 
 $\textbf{u}:[0,T]\times\mathbb{X}\to\mathbb{U}$ is a closed-loop (feedback) strategy and $\mathbb{U}$ is a separable metric space;
 \item the controlled diffusion process (state process) $X^{\textbf{u}}$ takes values in a state space $\mathbb{X}$, being either $\left(0,+\infty\right)$ or $\mathbb{R}^{n}$;
 \item $\mu: [0, T]\times\mathbb{X}\times\mathbb{U}\to\mathbb{R}^{n}$ and $\sigma: [0, T]\times\mathbb{X}\times\mathbb{U}\to\mathbb{R}^{n\times d}$ are measurable mappings,
where $n$ stands for the dimension of the state space $\mathbb{X}$. 
\end{itemize}

A feedback strategy $\textbf{u}$ is called \emph{state-independent} if it 
does not depend on the state $x$, i.e., it is a deterministic function of time. Moreover, a state-independent strategy $\textbf{u}$ is called \emph{continuous} if 
it is a continuous function of time.

In contrast to  standard stochastic control problems, the agent's objective function at time $t$, $J(t,x;\textbf{u})$, is defined implicitly as the solution of the following equation 
\begin{equation}\label{F}
	\mathbb{E}\left[K(X^{\textbf{u}}(T),J(t,x;\textbf{u}))|X^{\bu}(t)=x\right]=0,
\end{equation}
where $K:\mathbb{X}\times\mathbb{Z}\to\mathbb{R}$ with $\mathbb{Z}$ being either $\left(0,+\infty\right)$ or $\mathbb{R}$ is measurable and satisfies the following assumption.
\begin{assumption}\label{a1}
\begin{itemize}
\item[(i)] For every $x\in\mathbb{X}$, $K(x,z)$ is continuous and strictly decreasing in $z\in\mathbb{Z}$.
		
\item[(ii)] $K(x,z)$ is of polynomial growth in $x\in\X$, local-uniformly in $z\in\mathbb{Z}$.

\item[(iii)] For every $x\in\mathbb{X}$, $\inf\limits_{z\in\mathbb{Z}}K(x,z)<0<\sup\limits_{z\in\mathbb{Z}}K(x,z)$.
	\end{itemize}
\end{assumption}

For any feedback strategy $\textbf{u}$, let
\begin{align*}
	\mu^{\bu}(t,x):=\mu(t,x,\textbf{u}(t,x)),\, \sigma^{\bu}(t,x):=\sigma(t,x,\bu(t,x)),\,
	\Gamma^{\bu}(t,x):=\sigma(t,x,\bu(t,x))\sigma(t,x,\bu(t,x))^{\top}.
\end{align*}

\begin{definition}\label{d1}
	A feedback strategy $\mathbf{u}$ is called \emph{admissible} if the following three conditions hold:
	\begin{itemize}
		\item[(i)]  $\mu^{\bu}$ and $\sigma^{\bu}$ are locally Lipschitz in $x\in\mathbb{X}$, uniformly in $t\in[0, T]$.
		
		\item[(ii)] When $\mathbb{X}=\mathbb{R}^{n}$, $\mu^{\bu}$ and $\sigma^{\bu}$ are of linear growth in $x\in\mathbb{X}$, uniformly in $t\in[0, T]$; when $\mathbb{X}=\left(0,+\infty\right)$, $\mu^{\bu}$ and $\sigma^{\bu}$ have bounded norm in $x\in\mathbb{X}$, uniformly in $t\in[0, T]$.
		
		\item[(iii)] For each $x\in\X$, $\mu^{\bu}(t,x)$ and $\sigma^{\bu}(t,x)$ are right continuous in $t\in[0,T)$.	
	
\end{itemize}
Denote  by $\U$ the set of admissible strategies.
\end{definition}

In addition to guaranteeing the wellposedness of SDE (\ref{problem}), Conditions (i) and (ii) in Definition \ref{d1} lead to a favorable estimation of the solution; see Lemma \ref{lemmaonextimation}. Condition (iii) in Definition \ref{d1} imposes a mild continuity requirement, which will be used in the following study. 
\begin{lemma}\label{lemmaobj}
    Suppose that Assumption \ref{a1} holds. Then, for each $\left(t,x\right)\in[0,T)\times \X$ and $\bu\in\U$, there exists $z^{\bu}\in\mathbb{Z}$ such that 
		$f^{\bu}(t,x,z^{\bu})=0$, where $f^{\bu}$ is defined 
		by 
		\begin{equation}\label{fu}
	f^{\bu}(t,x,z):=\mathbb{E}\left[K(X^{\bu}(T),z)|X^{\bu}(t)=x\right],\quad z\in\mathbb{Z}.
\end{equation}
\end{lemma}
\begin{proof}
A combination Assumption \ref{a1}(ii) and Lemma \ref{lemmaonextimation} implies that the function $f^{\bu}$ 
is well defined by (\ref{fu}) and is of polynomial growth in $x\in\mathbb{X}$, local-uniformly in $z\in\mathbb{Z}$ and uniformly in $t\in[0,T]$. Obviously, $f^{\bu}$ is strictly decreasing in $z\in\mathbb{Z}$. The monotone convergence theorem implies that $f^{\bu}(t,x,z)$ is continuous in $z\in\mathbb{Z}$, $\inf_{z\in\mathbb{Z}}f^{\bu}(t,x,z)<0$ and $\sup_{z\in\mathbb{Z}}f^{\bu}(t,x,z)>0$. Then the conclusion follows.
\end{proof}

Lemma \ref{lemmaobj} implies that the objective function $J(t,x;\bu)$ is well defined by \eqref{F} for any $(t,x) \in[0,T)\times\X$ and $\bu\in\U$. Moreover, \eqref{F} reads $$f^{\bu}(t,x,J(t,x;\bu))=0.$$


We say that $\bu^{*}\in\U$ is optimal for $(t,x)$ if $\bu^{*}$ maximizes $J(t,x;\bu)$ over $\bu\in\U$. In the case that $K$ is separable, i.e., $K(x,z)=U_{1}(x)-U_{2}(z)$ for some functions $U_{1}$ and $U_{2}$ , where $U_{2}$ is a strictly increasing function, the objective function is equivalent to  $\mathbb{E}\left[U_{1}\left(X^{\bu}(T)\right)|X^{\bu}(t)=x\right]$ and the corresponding problem  is therefore time-consistent. In general, the  problem is time-inconsistent: the optimal strategy $\bu^{*}$ for $(t,x)$  is not necessarily optimal for $\left(s,X^{\bu^{*}}(s)\right)$, $s>t$.
Therefore, we consider the equilibrium solutions, instead of optimal solutions, following \cite{Bjork-2017,Hejiang2021}.

Hereafter, we always consider a fixed $\hat{\bu}\in\U$, which is a candidate equilibrium strategy. Moreover, we use $\D$\footnote{The concept of the alternative strategy set $\mathcal{D}$ originates from He and Jiang \cite{Hejiang2021}, aimed at investigating its impact on the characterization of equilibrium strategies. They demonstrate that as long as $\mathcal{D}$ contains the set of constant control strategies $\mathbb{U}$, the derived equilibrium conditions are independent of the specific choice of $\mathcal{D}$. We adopt this setup to study the corresponding problem within the framework of implicitly defined objective functions.}, which is a subset of $\U$, to denote the set of alternative strategies that at each time $t$ the agent can choose to implement in an infinitesimally small time period.
For given $t\in[0,T)$, $\epsilon\in(0,T-t)$, and $\ba\in\D$, let
\begin{equation}
	\bu_{t,\epsilon,\ba}(s,y):=\begin{cases}
		\ba(s,y),\quad s\in[t,t+\epsilon),\ y\in\X,\\
		\hat{\bu}(s,y),\quad s\notin[t,t+\epsilon),\ y\in\X.
	\end{cases}
\end{equation}
One can see that $\bu_{t,\epsilon,\ba}$ is admissible because  both of $\hat{\bu}$ and $\ba$ are admissible.

\begin{definition}{(Equilibrium Strategy)}
	We say that $\hat{\bu}$ is an equilibrium strategy if for every $(t,x)\in[0,T)\times\X$ and $\ba\in\D$, we have 
	\begin{equation}\label{eq}
		\limsup\limits_{\epsilon\downarrow0}\frac{J(t,x;\bu_{t,\epsilon,\ba})-J(t,x;\hat{\bu})}{\epsilon}\leq0.
	\end{equation}
\end{definition}

With a slight abuse of notation, each $u\in\mathbb{U}$ denotes the constant strategy $\bu\equiv u$. 
So $\mathbb{U}$ also stands for the set of all constant strategies.
We impose the following assumption throughout the paper.
\begin{assumption}\label{a2}
 $\mathbb{U}\subseteq\D\subseteq\U$.
\end{assumption}

\begin{remark}
The equilibrium stochastic control problem for implicitly defined objective functions is firstly investigated by \cite{liang2023dynamic} in the context of portfolio selection, where open-loop equilibrium controls are considered. In contrast,  closed-loop equilibrium strategies are considered in the present work. 
\end{remark}

\section{Characterization of equilibrium strategies}\label{sec:suffandness}

A difficulty in characterizing equilibrium strategies is that the objective function $J(t,x;\bu)$ is implicitly defined. This difficulty is overcome by the following lemma.

\begin{lemma}\label{lemma1}
	$\hat{\bu}$ is an equilibrium strategy if and only if, for any $(t,x)\in[0,T)\times\X$, $\ba\in\D$ and $\delta>0$, there exists some $\epsilon_{0}\in\left(0,T-t\right)$ such that
	\begin{equation}\label{lemmaieq1}
		\mathbb{E}\left[K(X^{\bu_{t,\epsilon,\ba}}(T),J(t,x;\hat{\bu})+\delta\epsilon)|X^{\bu_{t,\epsilon,\ba}}(t)=x\right]\leq0, \quad\forall\epsilon\in(0,\epsilon_{0}).
	\end{equation}
\end{lemma}
\begin{proof}
	It is easy to see that (\ref{eq}) holds if and only if, for any $\delta>0$, there exists some $\epsilon_{0}\in\left(0,T-t\right)$ such that
	\begin{equation*}
		J(t,x;\bu_{t,\epsilon,\ba})\leq J(t,x;\hat{\bu})+\delta\epsilon,\quad \forall\epsilon\in(0,\epsilon_{0}),
	\end{equation*}
i.e.,
\begin{equation*}
	f^{\bu_{t,\epsilon,\ba}}(t,x,J(t,x;\hat{\bu})+\delta\epsilon)\leq0, \quad \forall\epsilon\in(0,\epsilon_{0}),
\end{equation*}
which is equivalent to \eqref{lemmaieq1}.
\end{proof}

Now we introduce a generator of the controlled state process. Given $\bu\in\U$ and interval $[a,b]\subseteq[0,T)$, consider a function $\xi\in C^{1,2}([a,b]\times\X)$ and denote its first-order partial derivative with respect to $t$, the gradient vector with respect to $x$, and 
the Hessian matrix  with respect to $x$ by $\xi_{t}$, $\xi_{x}$, and $\xi_{xx}$, respectively. Let $\A^{\bu}\xi$ be defined by
\begin{equation*}
	\A^{\bu}\xi(t,x):=\xi_{t}(t,x)+\xi^{\top}_{x}(t,x)\mu^{\bu}(t,x)+\frac{1}{2}\tr\left(\xi_{xx}(t,x)\Gamma^{\bu}(t,x)\right),\quad (t,x)\in[a,b]\times\X.
\end{equation*}

 For the sake of convenience, we use $f$ to denote $f^{\hat{\bu}}$:
\begin{equation*}
	f(t,x,z):=f^{\hat{\bu}}(t,x,z)=\mathbb{E}\left[K(X^{\hat{\bu}}(T),z)\bigg|X^{\hat{\bu}}(t)=x\right], \quad (t,x,z)\in[0,T]\times\X\times\mathbb{Z}.
\end{equation*}
Moreover, $\A^{\bu}f$ denotes the function that is obtained by applying the operator $\A^{\bu}$ to $f$ as a function of $(t,x)$ while fixing $z$.

We need the following assumption on $f$ to study equilibrium strategies. 

\begin{assumption}\label{a3}
\begin{itemize}
\item[(i)]  For any $t\in[0,T)$, there exists $\tilde{t}\in(t,T)$ such that,  $f(\cdot,\cdot,z)\in C^{1,2}([t,\tilde{t}]\times\mathbb{X})$ for each fixed $z\in\mathbb{Z}$, and $\frac{\partial^{j+\alpha}f}{\partial t^{j}\partial x^{\alpha}}$ is of polynomial growth in $x\in\X$, local-uniformly in $z\in\mathbb{Z}$ and uniformly in $t'\in[t,\tilde{t}]$, for any derivative index $\alpha$ with $|\alpha|\leq 2-2j$ and $j=0,1$. 

\item[(ii)] For any $(t,x,z)\in[0,T)\times\X\times\mathbb{Z}$,  
there exists $\tilde{z}>z$, such that, for $\xi\in\{f_{t},\,f_{x},\,f_{xx}\}$,  $\xi(t',x',z')$ converges to $\xi(t,x,z')$ as $t'\downarrow t$ and $x'\rightarrow x$, uniformly in $z'\in[z,\tilde{z}]$, i.e., 
\begin{equation*}
	\lim\limits_{\substack{t\rightarrow t+ \\ x'\rightarrow x}}\sup\limits_{z'\in[z,\tilde{z}]}\lVert\xi(t',x',z')-\xi(t,x,z')\rVert=0.
\end{equation*}

\item[(iii)] For any $(t,x)\in[0,T)\times\X$, $f(t,x,z)$ is differentiable in $z\in\mathbb{Z}$, $f_{z}(t,x,z)<0$ for any $z\in\mathbb{Z}$, and $f_{t}(t,x,z)$, $f_{x}(t,x,z)$, $f_{xx}(t,x,z)$ are right continuous in  $z\in\mathbb{Z}$. 
\end{itemize}
\end{assumption}

The following lemma provides a sufficient condition of Assumption \ref{a3}(ii), which is easier to verify and will be used in proving Lemma \ref{lemma31}.

\begin{lemma}\label{remarktolemma}
	If for any $t\in[0,T)$, there exists $\tilde{t}\in(t,T)$ such that $\xi\in C\left([t,\tilde{t}]\times\X\times\mathbb{Z}\right)$, then $\xi$ satisfies Assumption \ref{a3}(ii).
\end{lemma}
\begin{proof}
See Appendix \ref{appRe3.1}.
\end{proof}

In the following, we always consider sufficiently small $\epsilon$, and $o(1)$ denotes a generic
function of $\epsilon$ such that $\lim\limits_{\epsilon\downarrow0}|o(1)|=0$. 
The next theorem characterizes equilibrium strategies. 

\begin{theorem}\label{sn}
	Suppose that Assumptions \ref{a1}, \ref{a2} and \ref{a3} hold. Then, for any $(t,x)\in[0,T)\times\X$, $\ba\in\mathcal{U}$ and $\delta\geq0$, we have 
	\begin{equation}\label{expansion0}
			\begin{split}
			&\mathbb{E}\left[K(X^{\bu_{t,\epsilon,\ba}}(T),J(t,x;\hat{\bu})+\delta\epsilon)|X^{\bu_{t,\epsilon,\ba}}(t)=x\right]\\=&\epsilon\left(\A^{\ba}f(t,x,J(t,x;\hat{\bu}))+\delta f_{z}(t,x,J(t,x;\hat{\bu}))\right)+\epsilon o(1).
		\end{split}
	\end{equation}
Moreover, $\hat{\bu}$ is an equilibrium strategy if and only if 
\begin{equation}\label{hjb}
\A^{u}f(t,x,J(t,x;\hat{\bu}))\leq0,\quad \forall u\in\mathbb{U},(t,x)\in[0,T)\times\X.
\end{equation}
\end{theorem}
\begin{proof} { Let $t\in[0,T)$ and $\tilde{t}\in(t,T)$.}
	By Assumption \ref{a3}(i), $f_{x}(t',x',z)$ is of polynomial growth in $x'\in\X$, local-uniformly in $z\in\mathbb{Z}$ and uniformly in $t'\in[t,\tilde{t}]$. Then, by Condition (ii) in Definition \ref{d1} and the estimation  in Lemma \ref{lemmaonextimation}, we have 
	\begin{equation}\label{key1}
		\mathbb{E}\left[\int_{t}^{\tilde{t}}\Vert f_{x}(s,X^{\ba}(s),z)\sigma^{\ba}(s,X^{\ba}(s))\Vert^{2}ds\bigg|X^{\ba}(t)=x\right]<+\infty,\quad \forall (x,z)\in\X\times\mathbb{Z}.
	\end{equation}
	
Similarly, we conclude from Assumption \ref{a3}(i) and Condition (ii) in Definition \ref{d1} that $\A^{\ba}f(t',x',z)$ is of polynomial growth in $x'\in\X$, local-uniformly in $z\in\mathbb{Z}$ and uniformly in $t'\in[t,\tilde{t}]$. As a result, { for any $[z,\tilde{z}]\subseteq\mathbb{Z}$ with $z<\tilde{z}$,} there exist constants $L>0$ and $L'>0$ and an integer $\gamma\geq1$ such that,  
	\begin{equation}\label{dct1}
		\begin{split}
			&\mathbb{E}\big[\sup\limits_{s\in[t,\tilde{t}]}\sup\limits_{z'\in[z,\tilde{z}]}\vert \A^{\ba}f(s,X^{\ba}(s),z')\vert\big|X^{\ba}(t)=x\big]\\\leq& L\mathbb{E}\big[\sup\limits_{s\in[t,\tilde{t}]}\left(1+ \varphi_{2\gamma}\left(X^{\ba}(s)\right)\right)\big|X^{\ba}(t)=x\big]\leq L'\left(1+ \varphi_{2\gamma}\left(x\right)\right),\quad\forall\,x\in\X,
		\end{split} 
	\end{equation}
where $\varphi_{2\gamma}$ is defined by (\ref{poly}).
	
{ Now we prove \eqref{expansion0}. Let $\delta\ge0$.  For any sufficiently small} $\epsilon\in\left(0,\tilde{t}-t\right)$, applying It\^{o}'s lemma to $f(s,X^{\ba}(s),z+\delta\epsilon)$,  $s\in[t,t+\epsilon]$, taking expectation, and recalling  Fubini's theorem, we conclude 
	\begin{equation}\label{ito}
		\begin{split}
&\mathbb{E}\left[f(t+\epsilon,X^{\ba}(t+\epsilon),z+\delta\epsilon)|X^{\ba}(t)=x\right]-f(t,x,z+\delta\epsilon)\\=&\int_{t}^{t+\epsilon}\mathbb{E}\left[\A^{\ba}f(s,X^{\ba}(s),z+\delta\epsilon)\bigg|X^{\ba}(t)=x\right]ds,
		\end{split}
	\end{equation}
where the expectation of the local martingale term is zero because, by  (\ref{key1}), the local martingale is a square-integrable  martingale. By Assumption \ref{a3}(i)-(ii), the continuity of $X^{\ba}$ with $X^{\ba}(t)=x$ and Conditions (i) and (iii) in Definition \ref{d1}, we have that $\A^{\ba}f(s,X^{\ba}(s),z')$ converges to $\A^{\ba}f(t,x,z')$ as $s\downarrow t$, uniformly in $z'\in[z,\tilde{z}]$. Then, according to (\ref{dct1}) and the dominated convergence theorem, we obtain 
\begin{equation}\label{key2}
\lim\limits_{s\downarrow t}\mathbb{E}\left[\sup\limits_{z'\in[z,\tilde{z}]}\big|\A^{\ba}f(s,X^{\ba}(s),z')-\A^{\ba}f(t,x,z')\big|\bigg|X^{\ba}(t)=x\right]=0.
\end{equation}
The above, together with (\ref{ito}) and Condition (iii) in Assumption \ref{a3}, yields 
\begin{equation}\label{expansion}
\mathbb{E}\left[f(t+\epsilon,X^{\ba}(t+\epsilon),z+\delta\epsilon)|X^{\ba}(t)=x\right]-f(t,x,z+\delta\epsilon)=\epsilon\A^{\ba}f(t,x,z)+\epsilon o(1).
\end{equation}
Consequently, when $\delta\geq0$, we have
\begin{equation*}
\begin{split}
&\mathbb{E}\left[K(X^{\bu_{t,\epsilon,\ba}}(T),J(t,x;\hat{\bu})+\delta\epsilon)|X^{\bu_{t,\epsilon,\ba}}(t)=x\right]\\=&\mathbb{E}\left[\mathbb{E}\big[K(X^{\bu_{t,\epsilon,\ba}}(T),J(t,x;\hat{\bu})+\delta\epsilon)|X^{\bu_{t,\epsilon,\ba}}(t+\epsilon)\big]|X^{\bu_{t,\epsilon,\ba}}(t)=x\right]\\=&\mathbb{E}\left[f(t+\epsilon,X^{\bu_{t,\epsilon,\ba}}(t+\epsilon),J(t,x;\hat{\bu})+\delta\epsilon)|X^{\ba}(t)=x\right]\\=&\mathbb{E}\left[f(t+\epsilon,X^{\ba}(t+\epsilon),J(t,x;\hat{\bu})+\delta\epsilon)|X^{\ba}(t)=x\right]\\
=&\mathbb{E}\left[f(t+\epsilon,X^{\ba}(t+\epsilon),J(t,x;\hat{\bu})+\delta\epsilon)|X^{\ba}(t)=x\right]-f(t,x,J(t,x;\hat{\bu})+\delta\epsilon)\\
&+f(t,x,J(t,x;\hat{\bu})+\delta\epsilon)-f(t,x,J(t,x;\hat{\bu}))\\=&\epsilon\A^{\ba}f(t,x,J(t,x;\hat{\bu}))+\epsilon\delta f_{z}(t,x,J(t,x;\hat{\bu}))+\epsilon o(1),
\end{split}
\end{equation*}
where the first equality is the case because of the Markov property of the state process and the tower property of conditional expectation,  the second equality is the case because $\bu_{t,\epsilon,\ba}=\hat{\bu}$ on $[t+\epsilon, T]$, the third is the case because $\bu_{t,\epsilon,\ba}=\ba$ on $[t, t+\epsilon)$ and the fourth is the case because $f(t,x,J(t,x;\hat{\bu}))=0$.

Now we prove the second part of the theorem. 
Assume that $\hat{\bu}$ is an equilibrium strategy. Then, 
by \eqref{expansion0} and Lemma \ref{lemma1}, for any $(t,x)\in[0,T)\times\X$, $\ba\in\D$ and $\delta>0$, we have
\begin{equation*}
\begin{split}
0\ge&\lim\limits_{\epsilon\downarrow0}\frac{1}{\epsilon}\mathbb{E}\left[K(X^{\bu_{t,\epsilon,\ba}}(T),J(t,x;\hat{\bu})+\delta\epsilon)|X^{\bu_{t,\epsilon,\ba}}(t)=x\right]\\
=&\A^{\ba}f(t,x,J(t,x;\hat{\bu}))+\delta f_{z}(t,x,J(t,x;\hat{\bu}))\\
=&\A^{\ba(t,x)}f(t,x,J(t,x;\hat{\bu}))+\delta f_{z}(t,x,J(t,x;\hat{\bu})).
	\end{split}
\end{equation*}
Letting  $\delta  \rightarrow 0$, we obtain
\begin{equation*}
	\A^{\ba(t,x)}f(t,x,J(t,x;\hat{\bu}))\leq0, \quad (t,x)\in[0,T)\times\X,\, \ba\in\D,
\end{equation*}
which yields (\ref{hjb}), by $\mathbb{U}\subseteq\D$. 

Conversely, assume that (\ref{hjb}) holds. Then, from (\ref{expansion0}) we obtain  for any $(t,x)\in[0,T)\times\X$, $a\in\D$ and $\delta>0$, 
\begin{equation*}
\begin{split}
&\mathbb{E}\left[K(X^{\bu_{t,\epsilon,\ba}}(T),J(t,x;\hat{\bu})+\delta\epsilon)|X^{\bu_{t,\epsilon,\ba}}(t)=x\right]\\
=&\epsilon\left(\A^{\ba(t,x)}f(t,x,J(t,x;\hat{\bu}))+\delta f_{z}(t,x,J(t,x;\hat{\bu}))\right)+\epsilon o(1)\\ \leq&\epsilon\left(\delta f_{z}(t,x,J(t,x;\hat{\bu}))+o(1)\right).
\end{split}
\end{equation*}
Hence, by $f_{z}(t,x,J(t,x;\hat{\bu}))<0$, there exists some $\epsilon_{0}\in(0,T-t)$ such that, for all $\epsilon\in(0,\epsilon_{0})$,
\begin{equation*}
	\mathbb{E}\left[K(X^{\bu_{t,\epsilon,\ba}}(T),J(t,x;\hat{\bu})+\delta\epsilon)|X^{\bu_{t,\epsilon,\ba}}(t)=x\right]\leq0.
\end{equation*}
Then, by Lemma \ref{lemma1}, $\hat{\bu}$ is an equilibrium strategy.
\end{proof}

\begin{remark}\label{remarkhjb}
Note that, under Assumption \ref{a2}, the characterization (\ref{hjb}) of equilibrium strategies does not depend on the choice of $\mathcal{D}$. 
    Considering the case $\mathcal{D}=\U$ and $\ba=\hat{\bu}$, we have
\begin{equation*}
	\mathbb{E}\left[F(X^{\bu_{t,\epsilon,\ba}}(T),J(t,x;\hat{\bu}))|X^{\bu_{t,\epsilon,\ba}}(t)=x\right]=\mathbb{E}\left[F(X^{\hat{\bu}}(T),J(t,x;\hat{\bu}))|X^{\hat{\bu}}(t)=x\right]=0.
\end{equation*}
Then, plugging it into (\ref{expansion0}) with $\delta=0$, we have
\begin{equation*}
\A^{\hat{\bu}(t,x)}f(t,x,J(t,x;\hat{\bu}))=\A^{\hat{\bu}}f(t,x,J(t,x;\hat{\bu}))=0.
\end{equation*}
Therefore, under Assumption \ref{a2}, Condition (\ref{hjb}) is equivalent to 
\begin{equation}\label{hjb2}
    \max\limits_{u\in\mathbb{U}}\A^{u}f(t,x,J(t,x;\hat{\bu}))=\A^{\hat{\bu}(t,x)}f(t,x,J(t,x;\hat{\bu}))=0,\quad (t,x)\in[0,T)\times\X.
\end{equation}
\end{remark}

\section{Applications}\label{sec:example}
In this section, we apply the main result in Section \ref{sec:suffandness} to investigate two portfolio selection problems for the CRRA betweenness preferences. 

\subsection{CRRA betweenness preferences}

Let the objective function $J(t,x;\bu)$ be implicitly defined by 
\begin{equation}\label{G1}
	\mathbb{E}\left[F\left(\frac{X^{\bu}(T)}{J(t,x;\bu)}\right)\bigg|X^{\bu}(t)=x\right]=0,
\end{equation}
where $F$ is a function satisfying the following assumption.

\begin{assumption}\label{a31}
$F:\left(0,+\infty\right)\rightarrow\mathbb{R}$ is twice continuously differentiable, $F(1)=0$, $F'(x)>0$ and $F''(x)<0$ for all $x\in\left(0,+\infty\right)$. Moreover, $F$, $F'$ and $F''$ are of polynomial growth in  $x\in\left(0,+\infty\right)$, i.e., there exist $L>0$ and  $\gamma>0$ such that 
$$\lvert F(x)\rvert+\lvert F'(x)\rvert+\lvert F''(x)\rvert\leq L\left(1+x^{\gamma}+ x^{-\gamma}\right),\quad x\in\left(0,+\infty\right).$$ 
\end{assumption}

We can see that (\ref{G1}) is a special case of (\ref{F}) with $K(x,z)=F(x/z)$, $(x,z)\in(0,+\infty)^2$. Obviously, the function $K$ given by this way satisfies Assumption \ref{a1}. The preference given by (\ref{G1}) has both of the betweenness property of \cite{Chew-1983,Chew-1989,Dekel-1986} and the positive homogeneity. Therefore, it is called a \emph{CRRA betweeness preference} (c.f. \cite{Backbook2017}). 

It is easy to see that any continuous state-independent strategy, i.e., any strategy $\bu$ satisfying $\bu(t,x)=v(t),\, (t,x)\in[0,T]\times(0,+\infty)$, with $v\in C([0,T])$, is admissible. 

Let us recall three important functions in \cite{liang2023dynamic}, the wellposedness of which has already been established there. Assume a random variable $\xi\sim\mathcal{N}(0,1)$.  Here and hereafter, $\mathcal{N}(0,1)$ denotes the standard normal distribution.

The first function $H:[0,+\infty)\rightarrow(0,+\infty)$ is implicitly defined by 
\begin{equation}
	\mathbb{E}\left[F\left(\frac{e^{\sqrt{y}\xi}}{H(y)}\right)\right]=0,\quad y\geq0.
\end{equation}
Obviously, $H$ is continuous and $H(0)=1$. Moreover, the following lemma shows  $H\in C^{1}[0,+\infty)$.  
\begin{lemma}\label{pH}
Under Assumption \ref{a31},	 the function $H$ is continuously differentiable on $[0,+\infty)$.
\end{lemma}
\begin{proof} See Appendix \ref{app:pH}.
\end{proof}

The second function $G:[0,+\infty)\rightarrow(0,+\infty)$ is defined by
\begin{equation}\label{g}
	G(y):=\frac{H(y)\mathbb{E}\left[e^{\sqrt{y}\xi}F'(e^{\sqrt{y}\xi}/H(y))\right]}{-\mathbb{E}\left[e^{2\sqrt{y}\xi}F''(e^{\sqrt{y}\xi}/H(y))\right]},\quad y\geq0,
\end{equation}
which is continuous on $[0,+\infty)$. 

\begin{remark}\label{Gbound}
If the index of relative risk aversion of $F$ has a strictly positive lower bound, i.e., there exists $c>0$ such that
\begin{equation*}
    -\frac{xF''(x)}{F'(x)}\geq c,\quad x\in(0,+\infty),
\end{equation*}
then we can see that $G\leq\frac{1}{c}$.    
\end{remark}

The following lemma presents a sufficient condition for $G$ to be continuously differentiable on $[0,+\infty)$, and hence, locally Lipschitz.

\begin{lemma}\label{sufflip}
   Suppose that Assumption \ref{a31} holds. If $F$  is four times continuously differentiable and $F^{(3)}$ and $F^{(4)}$ are of polynomial growth in  $x\in\left(0,+\infty\right)$, i.e., there exist $L>0$ and $\gamma>0$ such that 
$$ \lvert F^{(3)}(x)\rvert+\lvert F^{(4)}(x)\rvert\leq L\left(1+x^{\gamma}+ x^{-\gamma}\right),\quad x\in\left(0,+\infty\right),$$ 
then  $G$ is continuously differentiable on $[0,+\infty)$, hence, is locally Lipschitz.
\end{lemma}
\begin{proof}
See Appendix \ref{Appsufflip}.
\end{proof}

The third function $\mathcal{G}$ is defined by 
\begin{equation}
	\mathcal{G}(x):=\int_{0}^{x}\frac{1}{(G(y))^{2}}dy,\quad x\in[0,+\infty).
\end{equation}

\subsection{Financial market model} 

For simplicity, we consider the BS financial market model. We assume that there are one risk-free asset (bank account) and $d$ risky assets (stocks) in the market. The stock price processes $\left\{S_{i}(t),i=1,2,\cdots,d,\ t\in[0,T]\right\}$ follow the dynamics
\begin{equation*}
		dS_{i}(t)=S_{i}(t)\left[\mu_{i}(t)dt+\sigma_{i}(t)\cdot dW(t)\right],\quad t\in[0,T],\, i=1,\cdots,d,
\end{equation*}
where $W$ is a standard $d$-dimensional Brownian motion, and the market coefficients $\mu:[0,T]\rightarrow\mathbb{R}^{d}$ and $\sigma:[0,T]\rightarrow\mathbb{R}^{d\times d}$ as well as the interest rate $r:[0,T]\rightarrow\mathbb{R}$ of the bank account are continuous deterministic processes. Moreover, we assume that there exist some strictly positive constants $c_{1}$ and $c_{2}$ such that 
\begin{equation}\label{c}
	c_{2}\lVert\alpha\rVert^{2}\geq\lVert\sigma^{\top}(t)\alpha\rVert^{2}\geq c_{1}\lVert\alpha\rVert^{2},\quad \forall\alpha\in\mathbb{R}^{d},\ t\in[0,T].
\end{equation}
Let $\kappa(t):=\left(\sigma(t)\right)^{-1}\mu(t)$, 
$t\in[0,T]$, which will be used later.

\subsection{Portfolio selection with a constraint}\label{example1}
In this part, we consider a portfolio selection problem with a closed convex constraint. For simplicity, we assume that $r(t)=0$, $\forall t\in [0,T]$. Let the state space $\X=(0,+\infty)$ and $\mathbb{U}$ be a nonempty closed convex subset of $\mathbb{R}^{d}$.
Corresponding to a portfolio strategy (feedback strategy) $\bu:[0,T]\times(0,+\infty)\to \mathbb{U}$,  the self-financing wealth process (state process) $X^{\bu}$ satisfies the following SDE:
\begin{equation*}
	\begin{split}
		dX^{\bu}(s)=X^{\bu}(s)\left(\bu^{\top}(s,X^{\bu}(s))\mu(s)ds+\bu^{\top}(s,X^{\bu}(s))\sigma(s)dW(s)\right),\quad
		X^{\bu}(t)=x\in(0,+\infty),
	\end{split}
\end{equation*}
 where $x$ is the agent's wealth level at time $t$,
 $X^{\bu}(s)$ is the agent's wealth at time $s$, and $\bu\left(s,X^{\bu}(s)\right)$ is the proportion of the wealth invested into the stocks at time $s$. In this setting, we have
$$\mu^\bu(t,x)=x\bu^\top(t,x)\mu(t)\text{ and } \sigma^\bu(t,x)=x\bu^\top(t,x)\sigma(t),\quad (t,x)\in[0,T)\times(0,+\infty).$$
 Some examples of $\mathbb{U}$ can be found in \cite[Examples 5.4.1]{Karatzasbook01}, including no short-selling, incomplete market, prohibition of borrowing and so on.  We denote by $P_{t}$ the  projection onto the set $\sigma^{\top}(t)\mathbb{U}:=\left\{\sigma^{\top}(t)u:u\in\mathbb{U}\right\}$, i.e., $P_{t}(x)=\arg\min\limits_{z\in\sigma^{\top}(t)\mathbb{U}}\lVert x-z\rVert$ for $x\in\mathbb{R}^{d}$.

\begin{theorem}\label{CRRA}
	Suppose that Assumptions \ref{a2} and  \ref{a31} hold. 
	Let the objective function be given by \eqref{G1}.
	If the following ODE admits a unique solution on $[0,T]$,
	\begin{equation}\label{ode0}
		\begin{cases}
			A'(t)=-\lVert P_{t}\left(\kappa(t)G(A(t))\right)\rVert^{2},\quad t\in[0,T),\\
			A(T)=0,
		\end{cases}
	\end{equation}
then the strategy
	\begin{equation}\label{hatpi}
	\hat{\bu}(t,x)=(\sigma^{\top}(t))^{-1}P_{t}\left(\kappa(t)G(A(t))\right), \quad (t,x)\in[0,T]\times(0,+\infty),
\end{equation}
is the unique continuous state-independent equilibrium strategy.
\end{theorem}
\begin{proof}
See Appendix \ref{proofofCRRA}.
\end{proof}


\begin{remark}\label{rmk:ode}
Some comments on ODE (\ref{ode0}) are in order.
\begin{itemize}

\item[(i)] In the case $\mathbb{U}=\mathbb{R}^{d}$, i.e.,  there is no constraint on $\bu$, we recover ODE (4.13) in \cite{liang2023dynamic} as the following
	\begin{equation}\label{ode1}
	\begin{cases}
		A'(t)=-\lVert\left(\kappa(t)G(A(t))\right)\rVert^{2},\quad t\in[0,T),\\
		A(T)=0.
	\end{cases}
\end{equation}
By Lemma 4.3 in \cite{liang2023dynamic}, if $\mathcal{G}(\infty)>\int_{0}^{T}||\kappa(s)||^{2}ds$, then  ODE (\ref{ode1}) admits a unique solution on $[0,T]$, which is given by $A(t)=\mathcal{G}^{-1}\left(\int_{t}^{T}||\kappa(s)||^{2}ds\right)$. Moreover, \cite{liang2023dynamic} uses a verification argument to show that, when $\mathbb{U}=\mathbb{R}^{d}$, $\hat{\bu}$ given by (\ref{hatpi}) is an open-loop equilibrium. In Theorem \ref{CRRA}, we prove that $\hat{\bu}$ is also a closed-loop equilibrium and it is the unique equilibrium in the class of continuous state-independent strategies. It is consistent with the intuition that state-independent open-loop equilibrium and closed-loop equilibrium should coincide.

\item[(ii)] When the market coefficients $\mu$ and $\sigma$ are constant, the ODE reduces to the following
	\begin{equation}\label{ode2}
	\begin{cases}
		A'(t)=-\lVert P\left(\kappa G(A(t))\right)\rVert^{2},\quad t\in[0,T),\\
		A(T)=0,
	\end{cases}
\end{equation}
where $P(\cdot)$ is the projection onto the set $\sigma^{\top}\mathbb{U}:=\left\{\sigma^{\top}u: u\in\mathbb{U}\right\}$ and $\kappa=\sigma^{-1}\mu$. In this case, if $\mathcal{P}(\infty)>T$, where 
\begin{equation*}
	\mathcal{P}(x):=\int_{0}^{x}\frac{1}{\lVert P\left(\kappa G(y)\right)\rVert^{2}}dy,
\end{equation*}
 then the above ODE (\ref{ode2}) admits a unique solution on $[0,T]$, which is given by $A(t)=\mathcal{P}^{-1}\left(T-t\right)$.
 
 \item[(iii)] If $G$ defined by (\ref{g}) is independent of $y$, then ODE (\ref{ode0}) obviously admits a unique solution on $[0,T]$,  which is given by $A(t)=\int_{t}^{T}\lVert P_{s}\left(\kappa(s)G\right)\rVert^{2}ds$.

 \item[(iv)] The next lemma provides the wellposedness of ODE (\ref{ode0}) for a general case. For a sufficient condition for the locally Lipschitz condition of $G$, see Lemma \ref{sufflip}.
\end{itemize}
\end{remark}

\begin{lemma}\label{well}
Suppose that  Assumption \ref{a31} holds.
If $G$ defined by (\ref{g}) is locally Lipschitz  and the no-blow-up condition holds, i.e., $\mathcal{Q}^{-1}(T)<\infty$, where
 \begin{equation}\label{Q}
 	\mathcal{Q}(x):=\int_{0}^{x}\frac{1}{6c_{3}(G(y))^{2}+4c_{2}\lVert\beta\rVert^{2}}dy,
 \end{equation}
 $c_{3}=\max\limits_{t\in[0,T]}\lVert\kappa(t)\rVert^{2}$, $c_{2}$ is from (\ref{c}) and $\beta$ is a vector belongs to $\mathbb{U}$,
then ODE (\ref{ode0})  admits a unique solution on $[0,T]$.
\end{lemma}

\begin{proof} See Appendix \ref{AA}.
\end{proof}

\begin{example}[Weighted utility]\label{sec:exm:wu:c}
Let the objective function $J$ be given by (\ref{G1}) with  $F(x)=x^{1-\rho}x^{\gamma}-x^{\gamma}$, where $-1<\gamma<0$, $\gamma\leq \rho<\gamma+1$. Obviously, $F$ satisfies Assumption \ref{a31} and 
\begin{equation*}
	J(t,x;\bu)=\left(\frac{\mathbb{E}\left[\left(X^{\bu}(T)\right)^{1-\rho}\left(X^{\bu}(T)\right)^{\gamma}\bigg|X^{\bu}(t)=x\right]}{\mathbb{E}\left[\left(X^{\bu}(T)\right)^{\gamma}\bigg|X^{\bu}(t)=x\right]}\right)^{\frac{1}{1-\rho}},
\end{equation*}
which is also a special case of weighted utility (c.f. \cite{Backbook2017}). 
Moreover, we can see that  $$H(y)=e^{\frac{1}{2}(1-\rho+2\gamma)y}\text{ and } G(y)=\frac{1}{\rho-2\gamma},\quad y\ge0.$$ 
Then, by Remark \ref{rmk:ode}(iii), ODE (\ref{ode0}) admits a unique solution, which is given by
$A(t)=\int_{t}^{T}\lVert P_{s}\left(\frac{\kappa(s)}{\rho-2\gamma}\right)\rVert^{2}ds$.
Therefore, by Theorem \ref{CRRA},  
\begin{equation}\label{weighted}
		\hat{\bu}(t,x)=\left(\sigma^{\top}(t)\right)^{-1}P_{t}\left(\frac{\kappa(t)}{\rho-2\gamma}\right), \quad (t,x)\in[0,T]\times(0,+\infty),
	\end{equation}
is the unique continuous state-independent equilibrium strategy.
\end{example}

\begin{example}[Mixed CRRA utility]\label{sec:exm:crra:c}
Let the objective function $J$ be implicitly defined by
\begin{equation}\label{mixed}
	\mathbb{E}\left[\int_{\gamma_0}^{\gamma_1}U_{\gamma}\left(\frac{X^{\bu}(T)}{J(t,x;\bu)}\right)\textbf{F}(d\gamma)\bigg|X^{\bu}(t)=x\right]=0,
\end{equation}
where each $U_{\gamma}$ is a CRRA utility function\footnote{\label{footnotecrra}Keep in mind the difference between the CRRA utility functions and the CRRA betweenness preferences: a CRRA utility function is a function $U_\gamma$ given by \eqref{crrau}, whereas a CRRA betweenness preference is represented by \eqref{G1}.} given by
\begin{equation}\label{crrau}
	U_{\gamma}(x)=\begin{cases}
		\frac{x^{\gamma}-1}{\gamma},\quad \gamma\neq0,\gamma<1,\\
		\log x,\quad\gamma=0,
	\end{cases}
\end{equation}
$-\infty<\gamma_0<\gamma_1<1$, and $\textbf{F}$ is a probability measure on $\left(\gamma_0,\gamma_1\right)$. 
Note that (\ref{mixed}) is a special case of (\ref{G1}) with $F(x)=\int_{\gamma_0}^{\gamma_1} U_{\gamma}(x)\textbf{F}(d\gamma)$,  which satisfies Assumption \ref{a31}.
In this case, $H$ is implicitly defined by 
\begin{equation}\label{hmix}
   \mathbb{E}\left[\int_{\gamma_0}^{\gamma_1} U_{\gamma}\left(\frac{e^{\sqrt{y}\xi}}{H(y)}\right)\textbf{F}(d\gamma)\right]=0,\quad y\in[0,+\infty),
\end{equation}
and $G$ by
\begin{equation}\label{gmix}
	G(x)=\frac{\int_{\gamma_0}^{\gamma_1}(H(x))^{1-\gamma}e^{\frac{\gamma^{2}x}{2}}\textbf{F}(d\gamma)}{\int_{\gamma_0}^{\gamma_1}(1-\gamma)(H(x))^{1-\gamma}e^{\frac{\gamma^{2}x}{2}}\textbf{F}(d\gamma)},\quad x\in[0,+\infty).
\end{equation}
Lemma \ref{pH} shows that $H$ is continuously differentiable on $[0,+\infty)$.
By (\ref{gmix}), we know that $G$ is also continuously differentiable on $[0,+\infty)$, hence $G$ is locally Lipschitz. As $\frac{1}{1-\gamma_{0}}\leq G(x)\leq\frac{1}{1-\gamma_{1}}$ for any $x\geq0$, the no-blow-up condition holds. Therefore, by Lemma \ref{well},  ODE (\ref{ode0}) admits a unique solution $A$.
Then, by Theorem \ref{CRRA}, the strategy $\hat{\bu}$ given by \eqref{hatpi} is the unique continuous state-independent equilibrium strategy.
\end{example}

\begin{remark}
	(\ref{mixed}) is a natural generalization of the expected CRRA utility. Indeed, if $\textbf{F}$ is a Dirac measure $\delta_{\gamma}$, then $J(t,x;\bu)=U^{-1}_{\gamma}\left(\mathbb{E}\left[U_{\gamma}(X^{\bu}_{T})|X^{\bu}=x\right]\right)$. In this case, $\hat{\bu}(t,x)=(\sigma^{\top}(t))^{-1}P_{t}\left(\frac{\kappa(t)}{1-\gamma}\right)$, which coincides the solution of constrained portfolio optimization problem with CRRA utility. { See, for instance,   \cite[Theorem 14]{Hu-2005}}.
\end{remark}

\subsection{Portfolio selection with borrowing cost}\label{example2}
In this part, we consider a portfolio selection problem with borrowing cost: the interesting rate for borrowing money is higher than for saving. Let the state space $\X=(0,+\infty)$ and the control space $\mathbb{U}=\mathbb{R}^{d}$.
Corresponding to a feedback strategy $\bu:[0,T]\times(0,+\infty)\to \mathbb{U}$,  the self-financing wealth process (state process) $X^{\bu}$ satisfies the following SDE
 \begin{equation*}
	\begin{split}
		&dX^{\bu}(s)=X^{\bu}(s)\Big(\big(1-\textbf{1}\cdot\bu(s,X^{\bu}(s))\big)_+r(s)
		-\big(\textbf{1}\cdot\bu(s,X^{\bu}(s))-1\big)_+R(s)+\bu^\top(s,X^{\bu}(s))\mu(s)\Big)ds\\
		&\qquad\qquad\qquad +X^{\bu}(s)\bu^{\top}(s,X^{\bu}(s))\sigma(s)dW(s),\\
		&X^{\bu}(t)=x\in(0,+\infty),
	\end{split}
\end{equation*}
where $\textbf{1}$ denotes a $d$-dimensional column vector with all components being 1,  $r:[0,T]\rightarrow\mathbb{R}$ is a continuous function representing the interest rate for saving, $R:[0,T]\rightarrow\mathbb{R}$ is a continuous function representing the interesting rate for borrowing, which is \emph{no smaller} than $r$,  $x$ is the agent's wealth level at time $t$, $X^{\bu}(s)$ is the agent's wealth at time $s$,  $\bu\left(s,X^{\bu}(s)\right)$ is the proportion of the wealth invested into the stocks at time $s$, and  $a_+ =\max\{a,0\}$ for $a\in \mathbb{R}$. In this setting, we have, for $(t,x)\in[0,T)\times(0,+\infty)$,
\begin{equation*}
    \begin{split}
        &\mu^\bu(t,x)=x\big(\big(1-\textbf{1}\cdot\bu(t,x)\big)_+r(t)
		-\big(\textbf{1}\cdot\bu(t,x)-1\big)_+R(t)+\bu^\top(t,x)\mu(t)\big),\\ 
        &\sigma^\bu(t,x)=x\bu^\top(t,x)\sigma(t).
    \end{split}
\end{equation*}
 
We will consider the objective function $J$ which is implicitly defined by (\ref{G1}) with $F$ satisfying Assumption \ref{a31}. Similarly,  in this case any continuous state-independent strategy is also admissible. Let  $\kappa_{1}$ and $\kappa_{2}:[0,T]\rightarrow\mathbb{R}^{d}$ be defined by
\begin{equation*}
	\kappa_{1}(t)=\left(\sigma(t)\right)^{-1}\left(\mu(t)-r(t)\textbf{1}\right),\quad \kappa_{2}(t)=\left(\sigma(t)\right)^{-1}\left(\mu(t)-R(t)\textbf{1}\right).
\end{equation*}
\begin{theorem}\label{CRRA2}
	Suppose that Assumptions \ref{a2} and  \ref{a31} hold. If the following ODE admits a unique solution on $[0,T]$,
	\begin{equation}\label{ode.0}
		\begin{cases}
			A'(t)=-\mathcal{H}(t,A(t)),\quad t\in[0,T),\\
			A(T)=0,
		\end{cases}
	\end{equation}
where
\begin{equation*}
		\mathcal{H}(t,y)=\begin{cases}
		    \left(G(y)\right)^{2}\lVert\kappa_{2}(t)\rVert^{2},\quad\text{if}\,\,\, 1\leq G(y) \boldsymbol{1}^{T}(\sigma^{T}(t))^{-1}\kappa_{2}(t),\\
		   \bigg\lVert\left(\sigma(t)\right)^{-1}\bigg[G(y)\left(\mu(t)-\frac{\textbf{1}^{\top}\left(\sigma(t)\sigma^{\top}(t)\right)^{-1}\mu(t)}{\textbf{1}^{\top}\left(\sigma(t)\sigma^{\top}(t)\right)^{-1}\textbf{1}}\textbf{1}\right)+\frac{\textbf{1}}{\textbf{1}^{\top}\left(\sigma(t)\sigma^{\top}(t)\right)^{-1}\textbf{1}}\bigg]\bigg\rVert^{2},\\
           \qquad\qquad\qquad\text{if}\,\,\,G(y)\boldsymbol{1}^{T}(\sigma^{T}(t))^{-1}\kappa_{2}(t) < 1<G(y)\boldsymbol{1}^{T}(\sigma^{T}(t))^{-1}\kappa_{1}(t),\\
		   \left(G(y)\right)^{2}\lVert\kappa_{1}(t)\rVert^{2},\quad\text{if}\,\,\, G(y)\boldsymbol{1}^{T}(\sigma^{T}(t))^{-1}\kappa_{1}(t)\leq1,
		\end{cases}
\end{equation*}
	then the  strategy $\hat{\bu} $ defined by
\begin{equation}\label{hatpi2}
\hat{\bu}(t,x)=\begin{cases}
G(A(t))\left(\sigma^{\top}(t)\right)^{-1}\kappa_{2}(t),\quad\text{if}\,\,\, 1\leq G(A(t))\boldsymbol{1}^{T}(\sigma^{T}(t))^{-1}\kappa_{2}(t),\\
\left(\sigma(t)\sigma^{\top}(t)\right)^{-1}\bigg[G(A(t))\left(\mu(t)-\frac{\textbf{1}^{\top}\left(\sigma(t)\sigma^{\top}(t)\right)^{-1}\mu(t)}{\textbf{1}^{\top}\left(\sigma(t)\sigma^{\top}(t)\right)^{-1}\textbf{1}}\textbf{1}\right)+\frac{\textbf{1}}{\textbf{1}^{\top}\left(\sigma(t)\sigma^{\top}(t)\right)^{-1}\textbf{1}}\bigg],\\
\qquad\qquad\qquad\text{if}\,\,\, G(A(t))\boldsymbol{1}^{T}(\sigma^{T}(t))^{-1}\kappa_{2}(t) < 1<G(A(t))\boldsymbol{1}^{T}(\sigma^{T}(t))^{-1}\kappa_{1}(t),\\
G(A(t))\left(\sigma^{\top}(t)\right)^{-1}\kappa_{1}(t),\quad\text{if}\,\,\, G(A(t))\boldsymbol{1}^{T}(\sigma^{T}(t))^{-1}\kappa_{1}(t)\leq 1,
\end{cases}
\end{equation}
	is the unique continuous state-independent equilibrium strategy.
\end{theorem}
\begin{proof}
See Appendix \ref{proofCRRA2}.
\end{proof}
\begin{remark}\label{commentonode2}
If $G$ defined by (\ref{g}) is independent of $y$, then so is $\mathcal{H}$. Therefore, ODE (\ref{ode.0}) obviously admits a unique solution on $[0,T]$, which is given by 
$A(t)=\int_{t}^{T}\mathcal{H}(s)ds$.
Generally, the next lemma provides the wellposedness of ODE (\ref{ode.0}). For a sufficient condition for the locally Lipschitz condition and  global boundedness of $G$, see Lemma \ref{sufflip} and  Remark \ref{Gbound}.
\end{remark}

\begin{lemma}\label{well.}
   If $G$ defined by (\ref{g}) is locally Lipschitz  and bounded on $[0,+\infty)$,
		then  ODE (\ref{ode.0})  admits a unique solution on $[0,T]$.  
\end{lemma}

\begin{proof}
See Appendix \ref{BB}.
\end{proof}

\begin{example}[Weighted utility]
Consider the weighted utility as in Example \ref{sec:exm:wu:c}, from which we know 
$G(y)=\frac{1}{\rho-2\gamma}$, $y\ge0$.
Then by Remark \ref{commentonode2},  ODE (\ref{ode.0}) admits a unique solution. 
Let
	\begin{equation}\label{weighted2}
		\hat{\bu}(t,x)=\begin{cases}
			\frac{\left(\sigma^{\top}(t)\right)^{-1}\kappa_{2}(t)}{\rho-2\gamma},\quad\text{if}\,\,\, 1\leq\frac{1}{\rho-2\gamma}\boldsymbol{1}^{T}(\sigma^{T}(t))^{-1}\kappa_{2}(t),\\
			\left(\sigma(t)\sigma^{\top}(t)\right)^{-1}\bigg[\frac{1}{\rho-2\gamma}\left(\mu(t)-\frac{\textbf{1}^{\top}\left(\sigma(t)\sigma^{\top}(t)\right)^{-1}\mu(t)}{\textbf{1}^{\top}\left(\sigma(t)\sigma^{\top}(t)\right)^{-1}\textbf{1}}\textbf{1}\right)+\frac{\textbf{1}}{\textbf{1}^{\top}\left(\sigma(t)\sigma^{\top}(t)\right)^{-1}\textbf{1}}\bigg],\\
			\qquad\qquad\qquad\text{if}\,\,\, \frac{1}{\rho-2\gamma}\boldsymbol{1}^{T}(\sigma^{T}(t))^{-1}\kappa_{2}(t)< 1<\frac{1}{\rho-2\gamma}\boldsymbol{1}^{T}(\sigma^{T}(t))^{-1}\kappa_{1}(t),\\
			\frac{\left(\sigma^{\top}(t)\right)^{-1}\kappa_{1}(t)}{\rho-2\gamma},\quad\text{if}\,\,\, \frac{1}{\rho-2\gamma}\boldsymbol{1}^{T}(\sigma^{T}(t))^{-1}\kappa_{1}(t)\leq 1.
		\end{cases}
	\end{equation}
Then, by Theorem \ref{CRRA2}, $\hat{\bu}$ is the unique continuous state-independent equilibrium strategy.
\end{example}

\begin{example}[Mixed CRRA utility]
Let $F$ be generated by a mixture of CRRA utility as in (\ref{mixed}) and $G$ be given by (\ref{gmix}). Recalling from Example \ref{sec:exm:crra:c} that $G$ is locally Lipschitz and bounded on $[0,+\infty)$, by Lemma \ref{well.}, we know that ODE (\ref{ode.0}) admits a unique solution $A$. 
Then, by Theorem \ref{CRRA2}, the strategy $\hat{\bu}$ given by \eqref{hatpi2} is the unique continuous state-independent equilibrium strategy.
	In particular,  if $\textbf{F}$ is a Dirac measure $\delta_{\gamma}$, then $G(y)=\frac{1}{1-\gamma}$ and the equilibrium strategy reduces to the solution for CRRA  utility, i.e.,
		\begin{equation*}
		\hat{\bu}(t,x)=\begin{cases}
			\frac{\left(\sigma^{\top}(t)\right)^{-1}\kappa_{2}(t)}{1-\gamma},\quad\text{if}\,\,\, 1\leq\frac{1}{1-\gamma}\boldsymbol{1}^{T}(\sigma^{T}(t))^{-1}\kappa_{2}(t),\\
			\left(\sigma(t)\sigma^{\top}(t)\right)^{-1}\bigg[\frac{1}{1-\gamma}\left(\mu(t)-\frac{\textbf{1}^{\top}\left(\sigma(t)\sigma^{\top}(t)\right)^{-1}\mu(t)}{\textbf{1}^{\top}\left(\sigma(t)\sigma^{\top}(t)\right)^{-1}\textbf{1}}\textbf{1}\right)+\frac{\textbf{1}}{\textbf{1}^{\top}\left(\sigma(t)\sigma^{\top}(t)\right)^{-1}\textbf{1}}\bigg],\\
			\qquad\qquad\qquad\text{if}\,\,\, \frac{1}{1-\gamma}\boldsymbol{1}^{T}(\sigma^{T}(t))^{-1}\kappa_{2}(t)< 1<\frac{1}{1-\gamma}\boldsymbol{1}^{T}(\sigma^{T}(t))^{-1}\kappa_{1}(t),\\
			\frac{\left(\sigma^{\top}(t)\right)^{-1}\kappa_{1}(t)}{1-\gamma},\quad\text{if}\,\,\, \frac{1}{1-\gamma} \boldsymbol{1}^{T}(\sigma^{T}(t))^{-1}\kappa_{1}(t)\leq 1.
		\end{cases}
	\end{equation*}
	This coincides the results in \cite[Example 6.8.6]{Karatzasbook01} and \cite{Xu-1998}.
\end{example}

\section{Conclusions}\label{sec:conclusion}
In this paper we propose a class of stochastic control problems in a general diffusion framework, in which the objective function is implicitly defined. This kind of problems are inherently time-inconsistent. We provide a sufficient and necessary condition for a strategy to be a closed-loop equilibrium strategy. We apply this result to portfolio selection with CRRA betweenness preferences, allowing for convex constraints on the portfolio weights and borrowing cost.  The equilibrium strategies are explicitly characterized in terms of solutions of first-order ODEs for the case of deterministic market coefficients.
\vskip 10pt
{\bf Acknowledgements.}
This research is supported by the National Key R\&D Program of China
(NO. 2020YFA0712700) and NSFC (NOs. 12071146, 12271290, 11871036).

\begin{appendices}

\section{Estimation of the SDE (\ref{problem})}\label{appSDE}
For reader's convenience, we present an estimation of  SDE (\ref{problem}) under  Conditions (i) and (ii) in Definition \ref{d1}, which directly follows from Lemma A.2 in \cite{Hejiang2021}.

\begin{lemma}\label{lemmaonextimation}
     Under Conditions (i) and (ii) in Definition \ref{d1}, we have that for any $t\in[0,T)$ and $x\in\X$, there exists a unique strong solution to  SDE (\ref{problem}). Moreover, when $\X=\mathbb{R}^{n}$, for any integer $\gamma\geq1$, there exists a constant $L>0$ such that
     \begin{equation*}
         \mathbb{E}\bigg[\sup\limits_{s\in[t,T]}\lVert X(s)\rVert^{2\gamma}\bigg|X(t)=x\bigg]\leq L\left(1+\lVert x\rVert^{2\gamma}\right),\quad \forall t\in[0,T], x\in\X,
     \end{equation*}
     and when $\X=(0,+\infty)$, for any $\gamma\in\mathbb{R}$, there exists a constant $L>0$ such that 
     \begin{equation*}
         \mathbb{E}\bigg[\sup\limits_{s\in[t,T]}\lvert X(s)\rvert^{\gamma}\bigg|X(t)=x\bigg]\leq L\lvert x\rvert^{\gamma},\quad \forall t\in[0,T], x\in\X.
     \end{equation*}
\end{lemma}

\section{Some proofs}
\subsection{Proof of Lemma \ref{remarktolemma}}\label{appRe3.1}
 For any $(t,x,z)\in[0,T)\times\X\times\mathbb{Z}$, choosing $\tilde{z}>z$ and $r>0$ such that  $\overline{\mathbb{B}_{r}(x)}\in\X$, where $\overline{\mathbb{B}_{r}(x)}$ represents a bounded closed ball with $x$ as its center and $r$ as its radius. Obviously, $\xi$ is uniformly continuous on $[t,\tilde{t}]\times\overline{\mathbb{B}_{r}(x)}\times[z,\tilde{z}]$. Then, we claim that $\xi(t',x',z')$  converges to $\xi(t,x,z')$ as $t'\downarrow t$ and $x'\rightarrow x$, uniformly in $z'\in[z,\tilde{z}]$.
 
Indeed, let $\epsilon>0$, by uniform continuity, there is some $\delta>0$ such that  
$|\xi\left(t_{1},x_{1},z_{1}\right)-\xi\left(t_{2},x_{2},z_{2}\right)|<\epsilon$
 if $|t_{1}-t_{2}|+\lVert x_{1}-x_{2}\rVert+|z_{1}-z_{2}|<\delta$ and $\left(t_{1},x_{1},z_{1}\right),\left(t_{2},x_{2},z_{2}\right)\in [t,\tilde{t}]\times\overline{\mathbb{B}_{r}(x)}\times[z,\tilde{z}]$. 
Then 
$|\xi\left(t',x',z'\right)-\xi\left(t,x,z'\right)|<\epsilon$
if $(t'-t)+\lVert x'-x\rVert<\delta$, $t'>t$ and $z'\in[z,\tilde{z}]$.
 The above yields 
\begin{equation*}
   \sup\limits_{z'\in[z,\tilde{z}]} |\xi\left(t',x',z'\right)-\xi\left(t,x,z'\right)|<\epsilon,
\end{equation*}
which implies the desired uniform convergence.

\subsection{Proof of Lemma \ref{pH}}\label{app:pH}

Let $\xi\sim\mathcal{N}(0,1)$ and
\begin{equation*}
	h(y,z):=\mathbb{E}\left[F\left(\frac{e^{\sqrt{y}\xi}}{z}\right)\right],\quad y\geq0,\, z>0.
\end{equation*}
Clearly, $h(y,z)$ is strictly decreasing in $z\in(0,+\infty)$. By the dominated convergence theorem, $h$ is continuous on $[0,+\infty)\times(0,+\infty)$.  We are going to show that $h\in C^{1}\left([0,+\infty)\times(0,+\infty)\right)$ and $h_{z}(y,z)<0$ for each $(y,z)\in[0,+\infty)\times(0,+\infty)$, which combined with the implicit function theorem imply that $H$ is continuously differentiable on $[0,+\infty)$.
 By the dominated convergence theorem
\begin{equation*}
	\begin{split}
		&h_{y}(y,z)=\mathbb{E}\left[F'\left(\frac{e^{\sqrt{y}\xi}}{z}\right)\frac{e^{\sqrt{y}\xi}}{z}\frac{\xi}{2\sqrt{y}}\right],\quad{(y,z)\in(0,+\infty)\times(0,+\infty),}\\
		&h_{z}(y,z)=-\mathbb{E}\left[F'\left(\frac{e^{\sqrt{y}\xi}}{z}\right)\frac{e^{\sqrt{y}\xi}}{z^{2}}\right],\quad \quad{(y,z)\in[0,+\infty)\times(0,+\infty).}
	\end{split}
\end{equation*}
Obviously, $h_z<0$. { The continuity of $h_z$ is clear, by the polynomial growth of $F'$.}
It is left to show the continuity of $h_{y}$. Let us fix $\left(y_{0},z_{0}\right)\in(0,+\infty)\times(0,+\infty)$. Consider a set $\mathcal{I}:=[y_{1},y_{2}]\times[z_{1},z_{2}]$ such that $0<y_{1}<y_{0}<y_{2}<+\infty$ and $0<z_1<z_{0}<z_2<+\infty$. Then, by  the dominated convergence theorem, in order to prove $h_{y}$ is continuous at $\left(y_{0},z_{0}\right)$, it is enough to show that
\begin{equation*}
\mathbb{E}\left[\sup\limits_{(y,z)\in\mathcal{I}} F'\left(\frac{e^{\sqrt{y}\xi}}{z}\right)\frac{e^{\sqrt{y}\xi}}{z}\frac{\xi}{2\sqrt{y}}\right]	<+\infty,
\end{equation*}
 which is obvious by the polynomial growth of $F'$.
Then $h_{y}$ is continuous on $(0,+\infty)\times(0,+\infty)$ due to arbitrariness of $(y_{0},z_{0})$.

{
By the mean value theorem, for any $y>0$ and $z>0$,  there exists $y'\in(0,y)$ such that
\begin{equation*}
\begin{split}
      h_{y}(y^{2},z')=\frac{\mathbb{E}\left[F'\left(\frac{e^{y\xi}}{z'}\right)\frac{e^{y\xi}}{z'}\xi\right]}{2y}=\frac{1}{2}\mathbb{E}\left[F''\left(\frac{e^{y'\xi}}{z'}\right)\left(\frac{e^{y'\xi}}{z'}\right)^{2}\xi^{2}+F'\left(\frac{e^{y'\xi}}{z'}\right)\frac{e^{y'\xi}}{z'}\xi^{2}\right].
     \end{split}
\end{equation*}
Then, by the dominated convergence theorem,
\begin{equation*}
   \lim\limits_{(y,z')\rightarrow(0,z)}h_{y}(y^{2},z')=\frac{1}{2}\left(F''\left(\frac{1}{z}\right)\frac{1}{z^{2}}+F'\left(\frac{1}{z}\right)\frac{1}{z}\right).
\end{equation*}
Hence,
\begin{equation}\label{c0z}
   \lim\limits_{(y,z')\rightarrow(0,z)}h_{y}(y,z')=\frac{1}{2}\left(F''\left(\frac{1}{z}\right)\frac{1}{z^{2}}+F'\left(\frac{1}{z}\right)\frac{1}{z}\right).
\end{equation}
Similarly, by the mean value theorem, for any $y>0$ and $z>0$,  there exists $\eta\in(0,y)$ such that
\begin{equation}\label{con0}
    \frac{h(y,z)-h(0,z)}{y}=h_{y}(\eta,z).
\end{equation}
Then, letting $y\downarrow0$ yields 
\begin{equation*}
h_{y}(0,z)=\frac{1}{2}\left(F''\left(\frac{1}{z}\right)\frac{1}{z^{2}}+F'\left(\frac{1}{z}\right)\frac{1}{z}\right),
\end{equation*}
which, together with (\ref{c0z}), implies that
\begin{equation*}
    \lim\limits_{(y,z')\rightarrow(0,z)}h_{y}(y,z')=h_{y}(0,z).
\end{equation*}
Therefore, $h\in C^{1}\left([0,+\infty)\times(0,+\infty)\right)$.}

\subsection{Proof of Lemma \ref{sufflip}}\label{Appsufflip}
Recalling that $G$ is defined by \eqref{g}.
In order to prove that $G\in C^1([0,+\infty))$, it is enough to show that $g_{1}(y):=\mathbb{E}\left[e^{\sqrt{y}\xi}F'(e^{\sqrt{y}\xi}/H(y))\right]$ and $g_{2}(y):=\mathbb{E}\left[e^{2\sqrt{y}\xi}F''(e^{\sqrt{y}\xi}/H(y))\right]$ are  continuously differentiable on $[0,+\infty)$.  By the dominated convergence theorem, for each $y\in(0,+\infty)$,
\begin{equation*}
	\begin{split}
		&g_{1}'(y)=\mathbb{E}\left[e^{\sqrt{y}\xi}F'(e^{\sqrt{y}\xi}/H(y))\frac{\xi}{2\sqrt{y}}+e^{\sqrt{y}\xi}F''(e^{\sqrt{y}\xi}/H(y))\frac{e^{\sqrt{y}\xi}\frac{\xi}{2\sqrt{y}}H(y)-e^{\sqrt{y}\xi}H'(y)}{H^{2}(y)}\right],\\
		&g_{2}'(y)=\mathbb{E}\left[e^{2\sqrt{y}\xi}F''(e^{\sqrt{y}\xi}/H(y))\frac{\xi}{\sqrt{y}}+e^{2\sqrt{y}\xi}F^{(3)}(e^{\sqrt{y}\xi}/H(y))\frac{e^{\sqrt{y}\xi}\frac{\xi}{2\sqrt{y}}H(y)-e^{\sqrt{y}\xi}H'(y)}{H^{2}(y)}\right].
	\end{split}
\end{equation*}
Now we prove the continuity of $g_{2}'$. Let us fix $y_{0}\in(0,+\infty)$. Consider a set $\mathcal{I}:=[y_{1},y_{2}]$ such that $0<y_{1}<y_{0}<y_{2}<+\infty$. Then, by the dominated convergence theorem, in order to prove that $g_{2}'$ is continuous at $y_{0}$, it is enough to show that
\begin{equation*}
\begin{split}
    &\mathbb{E}\left[\sup\limits_{y\in\mathcal{I}}e^{2\sqrt{y}\xi}\bigg|F''(e^{\sqrt{y}\xi}/H(y))\frac{\xi}{\sqrt{y}}\bigg|\right]<+\infty,\\
     &\mathbb{E}\left[\sup\limits_{y\in\mathcal{I}}e^{3\sqrt{y}\xi}\bigg|F^{(3)}(e^{\sqrt{y}\xi}/H(y))\frac{\frac{\xi}{2\sqrt{y}}}{H(y)}\bigg|\right]<+\infty,\\
      &\mathbb{E}\left[\sup\limits_{y\in\mathcal{I}}e^{3\sqrt{y}\xi}\bigg|F^{(3)}(e^{\sqrt{y}\xi}/H(y))\frac{H'(y)}{H^{2}(y)}\bigg|\right]<+\infty,
\end{split}
\end{equation*}
 which are obvious by the polynomial growth of $F''$ and $F^{(3)}$.
Then $g_{2}'$ is continuous on $(0,+\infty)$ due to the arbitrariness of $y_{0}$. 

Note that
\begin{equation*}
 \lim\limits_{y\downarrow0}g_{2}'(y)= \lim\limits_{y\downarrow0}\frac{\mathbb{E}\left[e^{2\sqrt{y}\xi}F''(e^{\sqrt{y}\xi}/H(y))\xi\right]}{\sqrt{y}}+\lim\limits_{y\downarrow0}\frac{\mathbb{E}\left[e^{3\sqrt{y}\xi}F^{(3)}(e^{\sqrt{y}\xi}/H(y))\xi\right]}{2\sqrt{y}}-F^{(3)}(1)H'(0).
\end{equation*}
  Using L'H\^{o}pital's rule, we obtain  
  \begin{equation*}
  \begin{split}
         &\lim\limits_{y\downarrow0}\frac{\mathbb{E}\left[e^{2\sqrt{y}\xi}F''(e^{\sqrt{y}\xi}/H(y))\xi\right]}{\sqrt{y}}\\=& \lim\limits_{y\downarrow0}\frac{\mathbb{E}\left[e^{2\sqrt{y}\xi}F''(e^{\sqrt{y}\xi}/H(y))\xi^{2}\frac{1}{\sqrt{y}}\right]}{\frac{1}{2\sqrt{y}}}+\lim\limits_{y\downarrow0}\frac{\mathbb{E}\left[e^{3\sqrt{y}\xi}F^{(3)}(e^{\sqrt{y}\xi}/H(y))\xi^{2}\frac{1}{2\sqrt{y}}\right]}{\frac{1}{2\sqrt{y}}}\\=&2F''(1)+F^{(3)}(1).
  \end{split}
  \end{equation*}
Similarly, 
\begin{equation*}
    \lim\limits_{y\downarrow0}\frac{\mathbb{E}\left[e^{3\sqrt{y}\xi}F^{(3)}(e^{\sqrt{y}\xi}/H(y))\xi\right]}{2\sqrt{y}}=\frac{3}{2}F^{(3)}(1)+\frac{1}{2}F^{(4)}(1).
\end{equation*}
Arguing as (\ref{con0}), we have  
\begin{equation*}
    g_{2}'(0)=\lim\limits_{y\downarrow0}g_{2}'(y).
\end{equation*}
Therefore, $g_{2}\in C^{1}[0,+\infty)$. Similarly, one can  prove  $g_{1}\in C^{1}[0,+\infty)$.

\subsection{Proof of Theorem \ref{CRRA}}\label{proofofCRRA}
We now consider the following candidate state-independent strategy
\begin{equation}\label{eq:bu:a}
	\hat{\bu}(t,x)=\left(\sigma^{\top}(t)\right)^{-1}a(t),\quad (t,x)\in[0,T]\times(0,+\infty),
\end{equation}
where $a$ is a continuous function. 

\begin{lemma}\label{lemma31}
	Suppose that Assumptions \ref{a31} holds. Let $\hat{\bu}$ be given by \eqref{eq:bu:a}. Then, the following function 
	\begin{equation}
		f(t,x,z)=\mathbb{E}\left[F\left(\frac{X^{\hat{\bu}}(T)}{z}\right)\bigg|X^{\hat{\bu}}(t)=x\right],\quad (t,x,z)\in[0,T]\times(0,+\infty)\times(0,+\infty),
	\end{equation}
satisfies Assumption \ref{a3}. Moreover, we have  for each $(t,x,z)\in[0,T)\times(0,+\infty)\times(0,+\infty)$,
\begin{equation}\label{diff}
	\begin{split}
		&f_{x}(t,x,z)=\frac{1}{x}\mathbb{E}\left[F'\left(\frac{X^{\hat{\bu}}(T)}{z}\right)\frac{X^{\hat{\bu}}(T)}{z}\bigg|X^{\hat{\bu}}(t)=x\right],\\
		&f_{xx}(t,x,z)=\frac{1}{x^{2}}\mathbb{E}\left[F''\left(\frac{X^{\hat{\bu}}(T)}{z}\right)\left(\frac{X^{\hat{\bu}}(T)}{z}\right)^{2}\bigg|X^{\hat{\bu}}(t)=x\right].
	\end{split}
\end{equation}
\end{lemma}
\begin{proof}
   For the strategy $\hat{\bu}$ given by \eqref{eq:bu:a}, we have
   \begin{equation*}
   	X^{\hat{\bu}}(T)=X^{\hat{\bu}}(t) e^{B(t)+\int_{t}^{T}a^{\top}(s)dW(s)},
   \end{equation*}
where 
\begin{equation}\label{eq:AB}
A(t):=\int_{t}^{T}||a(s)||^{2}ds\quad\text{and}\quad
B(t):=\int_{t}^{T}a^{\top}(s)\kappa(s)ds-\frac{1}{2}A(t).
\end{equation}
Then, we obtain
\begin{equation*}
	f(t,x,z)=\mathbb{E}\left[F\left(\frac{x}{z}e^{B(t)+\sqrt{A(t)}\xi}\right)\right],
\end{equation*}
where $\xi\sim\mathcal{N}(0,1)$.

Let $t_{0}=\inf\{t\in[0,T]\mid A(t)=0\}$. Obviously, $A(t)=0$ and $f(t,x,z)=F(\frac{x}{z})$, for any $(t,x,z)\in[t_{0},T]\times(0,+\infty)\times(0,+\infty)$. The verification of Assumption \ref{a3} for $t\in[t_{0},T)$ is trivial. Therefore, we only need to consider the case $t_0>0$ and those $t\in[0,t_0)$ in the verification of Assumption \ref{a3}.

\paragraph{Verification of Assumption \ref{a3} (i)} 
Assume $t\in[0,t_0)$. Let $\tilde{t}\in(t,t_0)$.  We claim that, for each $(s,x,z)\in[t,\tilde{t}]\times(0,+\infty)\times(0,+\infty)$,
\begin{equation}\label{diffft}
f_{t}(s,x,z)=\mathbb{E}\left[F'\left(\frac{x}{z} e^{B(s)+\sqrt{A(s)}\xi}\right)\left(\frac{x}{z} e^{B(s)+\sqrt{A(s)}\xi}\right)\left(B'(s)+\frac{1}{2}\frac{A'(s)}{\sqrt{A(s)}}\xi\right)\right].
\end{equation}
To prove this claim, according to the dominated convergence theorem, it is enough to show that, for each $(x,z)\in(0,+\infty)^{2}$,
\begin{equation*}
C(x,z):=\mathbb{E}\left[\sup_{s\in[t,\tilde{t}]}F'\left(\frac{x}{z} e^{B(s)+\sqrt{A(s)}\xi}\right)\left(\frac{x}{z} e^{B(s)+\sqrt{A(s)}\xi}\right)\left|B'(s)+\frac{1}{2}\frac{A'(s)}{\sqrt{A(s)}}\xi\right|\right]<+\infty.
\end{equation*}
Indeed, using the Cauchy-Schwarz inequality, we have
\begin{equation*}
	\begin{split}
	C(x,z)\leq&\mathbb{E}\left[\sup_{s\in[t,\tilde{t}]}\left(F'\left(\frac{x}{z} e^{B(s)+\sqrt{A(s)}\xi}\right)\right)^{2}\left(\frac{x}{z} e^{B(s)+\sqrt{A(s)}\xi}\right)^{2}\right]^{\frac{1}{2}}\\&\times\mathbb{E}\left[\sup_{s\in[t,\tilde{t}]}\left(B'(s)+\frac{1}{2}\frac{A'(s)}{\sqrt{A(s)}}\xi\right)^{2}\right]^{\frac{1}{2}}.
	\end{split}
\end{equation*}
Obviously, 
\begin{equation}\label{ineq2}
	\mathbb{E}\left[\sup_{s\in[t,\tilde{t}]}\left(B'(s)+\frac{1}{2}\frac{A'(s)}{\sqrt{A(s)}}\xi\right)^{2}\right]<+\infty,
\end{equation}
and  we conclude from Assumption \ref{a31} that
\begin{equation}\label{ineqa1}
	\begin{split}
		&\mathbb{E}\left[\sup_{s\in[t,\tilde{t}]}\left(F'\left(\frac{x}{z} e^{B(s)+\sqrt{A(s)}\xi}\right)\right)^{2}\left(\frac{x}{z} e^{B(s)+\sqrt{A(s)}\xi}\right)^{2}\right]\\\leq& C_1\left(1+\left(\frac{x}{z}\right)^{\gamma}+\left(\frac{x}{z}\right)^{-\gamma}\right)^{2}\left(\frac{x}{z}\right)^{2}\mathbb{E}\left[\left(e^{\gamma C_1\xi}+1\right)^{2}\left(e^{2C_1\xi}+1\right)^{2}\right]<+\infty,
	\end{split}
\end{equation}
where $C_1>0$ is a constant. Hence, the claim follows.
  Moreover, by (\ref{ineq2}) and (\ref{ineqa1}) we conclude that $f_{t}$ is of polynomial growth in $x\in(0,+\infty)$, local-uniformly in $z\in(0,+\infty)$ and uniformly in $s\in[t,\tilde{t}]$.  We will show  $f_{t}\in C\left([t,\tilde{t}]\times(0,+\infty)\times(0,+\infty)\right)$ in the next paragraph, which, together with the above,  implies that $f_{t}$ satisfies Assumption \ref{a3}(i).  Similar arguments apply to the verification for $f_{x}$ and $f_{xx}$.
  
  \paragraph{Verification of Assumption \ref{a3} (ii)}
  For any $(s,x,z)\in[t,\tilde{t}]\times(0,+\infty)\times(0,+\infty)$, consider the set $\mathcal{I}:=[x_1,x_2]\times[z_1,z_2]$ such that 
  $0<x_1<x<x_2<+\infty$ and $0<z_1<z<z_2<+\infty$. In a similar way as in the previous paragraph, we know  
  \begin{equation*}
      \mathbb{E}\left[\sup\limits_{(s,x,z)\in[t,\tilde{t}]\times\mathcal{I}}F'\left(\frac{x}{z}e^{B(s)+\sqrt{A(s)}\xi}\right)\left(\frac{x}{z} e^{B(s)+\sqrt{A(s)}\xi}\right)\left|B'(s)+\frac{1}{2}\frac{A'(s)}{\sqrt{A(s)}}\xi\right|\right]<+\infty.
  \end{equation*}
  Then, the dominated convergence theorem implies that $f_{t}$ is continuous at $(s,x,z)$. Therefore, $f_{t}\in C\left([t,\tilde{t}]\times(0,+\infty)\times(0,+\infty)\right)$ by the arbitrariness of $(s,x,z)$. Consequently,  by Lemma \ref{remarktolemma}, $f_{t}$ satisfies  Assumption \ref{a3}(ii). Similar arguments apply to the verification for $f_{x}$ and $f_{xx}$.
  
\paragraph{Verification of Assumption \ref{a3} (iii)} Similar to (\ref{diffft}), we can obtain, for each $(t,x,z)\in[0,T)\times(0,+\infty)\times(0,+\infty)$,
\begin{equation*}
    f_{z}(t,x,z)=-\frac{1}{z^{2}}\mathbb{E}\left[F'\left(\frac{X^{\hat{\bu}}(T)}{z}\right)X^{\hat{\bu}}(T)\bigg|X^{\hat{\bu}}(t)=x\right],
\end{equation*}
which is obviously negative. Moreover, as  $f_{t}\in C\left([t,\tilde{t}]\times(0,+\infty)\times(0,+\infty)\right)$ for any $t\in[0,T)$ with $\tilde{t}\in(t,T)$, we have that $f_{t}(t,x,z)$ is right continuous in $z\in(0,+\infty)$ for any $(t,x)\in[0,T)\times(0,+\infty)$. Similar arguments apply to show the right continuity of $f_{x}$ and $f_{xx}$ in $z\in(0,+\infty)$.

Finally, the proof of (\ref{diff}) is similar to (\ref{diffft}), and we omit it here.
\end{proof}

\begin{lemma}\label{lma:GA}
Suppose that Assumptions \ref{a31} holds. Let $\hat{\bu}$ be given by \eqref{eq:bu:a} and $A$ be given by \eqref{eq:AB}.
Then
$$-\frac{f_{x}(t,x,J(t,x;\hat{\bu}))}{xf_{xx}(t,x,J(t,x;\hat{\bu}))}=G(A(t)), \quad \forall\,(t,x)\in[0,T)\times(0,+\infty).
$$
\end{lemma}
\begin{proof}
Based on the definition of the objective function, we have
\begin{equation*}
0=\mathbb{E}\left[F\left(\frac{X^{\hat{\bu}}(T)}{J(t,x;\hat{\bu})}\right)\bigg|X^{\hat{\bu}}(t)=x\right]=\mathbb{E}\left[F\left(\frac{xe^{B(t)+\sqrt{A(t)}\xi}}{J(t,x;\hat{\bu})}\right)\right],
\end{equation*}
where $A$ and $B$ are  given by \eqref{eq:AB}
and $\xi\sim\mathcal{N}(0,1)$. Consequently,
\begin{equation}\label{ph}
J(t,x;\hat{\bu})=xe^{B(t)}H(A(t)),\quad
(t,x)\in[0,T]\times(0,+\infty).
\end{equation}
A combination of (\ref{diff}) and (\ref{ph}) yields  
\begin{equation*}\label{GA}
    \begin{split}
        -\frac{f_{x}(t,x,J(t,x;\hat{\bu}))}{xf_{xx}(t,x,J(t,x;\hat{\bu}))}=\frac{H(A(t))\mathbb{E}\left[e^{\sqrt{A(t)}\xi}F'\left(e^{\sqrt{H(A(t))}\xi}/H(A(t))\right)\right]}{-\mathbb{E}\left[e^{2\sqrt{A(t)}\xi}F''\left(e^{\sqrt{H(A(t))}\xi}/H(A(t))\right)\right]}=G(A(t))
    \end{split}
\end{equation*}
for all $(t,x)\in[0,T)\times(0,+\infty)$.
\end{proof}

We are now ready to present the proof of the Theorem \ref{CRRA}.

\begin{proof}[Proof of Theorem \ref{CRRA}]
Let $\hat{\bu}$ be given by \eqref{eq:bu:a}.
By Lemma \ref{lemma31}, $f$ satisfies Assumption \ref{a3}.
Then by Theorem \ref{sn} and Remark \ref{remarkhjb},  $\hat{\bu}$ is an equilibrium strategy if and only if
	\begin{align}
		\nonumber &a(t)=\arg\max\limits_{a\in\sigma^{\top}(t)\mathbb{U}}\Big\{f_{t}(t,x,J(t,x;\hat{\bu}))+a^{\top}\kappa(t)xf_{x}(t,x,J(t,x;\hat{\bu}))\\
		&\qquad\qquad\qquad\qquad+\frac{1}{2}\lVert a\rVert^{2}x^{2}f_{xx}(t,x,J(t,x;\hat{\bu}))\Big\}, \quad t\in[0,T),	\label{pi}\\
\nonumber	&f_{t}(t,x,J(t,x;\hat{\bu}))+a^{\top}(t)\kappa(t) xf_{x}(t,x,J(t,x;\hat{\bu}))\\
&\qquad\qquad+\frac{1}{2}\lVert a(t)\rVert^{2}x^{2}f_{xx}(t,x,J(t,x;\hat{\bu}))=0,\quad (t,x)\in[0,T)\times(0,+\infty).\label{fc}
\end{align}

 We claim that (\ref{fc}) always holds. Indeed, by Lemma \ref{lemma31}, we know that $f$ satisfies Assumption \ref{a3}(i). Then, for any sufficiently small $\epsilon>0$, 
\begin{equation*}
\begin{split}
     f(t,x,J(t,x;\hat{\bu}))=\mathbb{E}\left[f(t+\epsilon,X^{\hat{\bu}}(t+\epsilon),J(t,x;\hat{\bu}))\bigg|X^{\hat{\bu}}(t)=x\right],
\end{split}
\end{equation*}
which implies that by arguing as (\ref{ito})
\begin{equation*}
    \int_{t}^{t+\epsilon}\mathbb{E}\left[\A^{\hat{\bu}}f(s,X^{\hat{\bu}}(s),J(t,x;\hat{\bu}))\bigg|X^{\hat{\bu}}(t)=x\right]ds=0.
\end{equation*}
Dividing $\epsilon$ and then letting $\epsilon\downarrow0$ yields (\ref{fc}).

Therefore, it is enough to consider equation (\ref{pi}). Note that for each $(t,x)\in[0,T)\times(0,+\infty)$, 
\begin{equation*}
\begin{split}
&f_{t}(t,x,J(t,x;\hat{\bu}))+a^{\top}\kappa(t)xf_{x}(t,x,J(t,x;\hat{\bu}))+\frac{1}{2}\lVert a\rVert^{2}x^{2}f_{xx}(t,x,J(t,x;\hat{\bu}))\\
=&\frac{1}{2}x^{2}f_{xx}(t,x,J(t,x;\hat{\bu}))\bigg\lVert a+\frac{f_{x}(t,x,J(t,x;\hat{\bu}))}{xf_{xx}(t,x,J(t,x;\hat{\bu}))}\kappa(t)\bigg\rVert^{2}+f_{t}(t,x,J(t,x;\hat{\bu}))\\
-&\frac{1}{2}\frac{\left(xf_{x}(t,x,J(t,x;\hat{\bu}))\right)^{2}}{x^{2}f_{xx}(t,x,J(t,x;\hat{\bu}))}\lVert\kappa(t)\rVert^{2}.
\end{split}
\end{equation*}
As $F''(x)<0$ for all $x\in(0,+\infty)$, we conclude from (\ref{diff}) that $f_{xx}(t,x,J(t,x;\hat{\bu}))<0$. Hence, we know that \eqref{pi} is equivalent to
\begin{equation}\label{pi2}
	\begin{split}
		a(t)&=P_{t}\left(-\frac{f_{x}(t,x,J(t,x;\hat{\bu}))}{xf_{xx}(t,x,J(t,x;\hat{\bu}))}\kappa(t)\right),
	\end{split}
\end{equation}
which, by Lemma \ref{lma:GA}, is further equivalent to
\begin{equation*}
	\begin{split}
		a(t)&=P_{t}\left(\kappa(t)G(A(t))\right).
	\end{split}
\end{equation*}
Hence, we know that $A$ satisfies ODE \eqref{ode0}, 
which admits a unique solution on $[0,T]$ by assumption. Therefore, $\hat{\bu}$ defined by (\ref{hatpi}) is the unique continuous state-independent equilibrium strategy.
\end{proof}

\subsection{Proof of Lemma \ref{well}}\label{AA}
	Note that  ODE (\ref{ode0}) is equivalent to the following ODE
\begin{equation}\label{ode3}
	\begin{cases}
		A'(t)=\lVert P_{T-t}\left(\kappa(T-t)G(A(t))\right)\rVert^{2},\quad t\in(0,T],\\
		A(0)=0.
	\end{cases}
\end{equation}
As the projection $P$ is a contraction, $\kappa$ is bounded and $G$ is locally Lipschitz, it follows that $$\lVert P_{T-t}\left(\kappa(T-t)G(A)\right)\rVert^{2}$$ is also locally Lipschitz in $A\in[0,+\infty)$, uniformly in $t\in[0,T]$.
Therefore, by \cite[Theorem 7.4]{Amannbook90}, the local wellposedness of  ODE (\ref{ode3}) is obtained. 

Now, we show that the local solution does not blow up under the condition that $\mathcal{Q}^{-1}(T)<\infty$. Note that 
\begin{equation*}
	\begin{split}
			\lVert P_{T-t}\left(\kappa(T-t)G(A(t))\right)\rVert^{2}&\leq 2\lVert P_{T-t}\left(\kappa(T-t)G(A(t))\right)-\kappa(T-t)G(A(t))\rVert^{2}\\
			&+ 2\lVert \kappa(T-t)G(A(t))\rVert^{2}\\&\leq 2\lVert\sigma^{\top}(T-t)\beta -\kappa(T-t)G(A(t))\rVert^{2}+ 2\lVert \kappa(T-t)G(A(t))\rVert^{2}\\&\leq
			6c_{3}(G(A(t)))^{2}+4c_{2}\lVert\beta\rVert^{2}.
	\end{split}
\end{equation*}
Then 
\begin{equation*}
	\begin{split}
		A'(t)\leq 6c_{3}(G(A(t)))^{2}+4c_{2}\lVert\beta\rVert^{2}, \quad t\in(0,T],
	\end{split}
\end{equation*}
yielding  
\begin{equation*}
	\mathcal{Q}(A(t))\leq t,\quad t\in[0,T].
\end{equation*}
As $\mathcal{Q}$ is strictly increasing, we have 
\begin{equation*}
	A(t)\leq \mathcal{Q}^{-1}(t)<\mathcal{Q}^{-1}(T)<+\infty, \quad t\in[0,T],
\end{equation*}
which guarantees the grobal wellposedness of ODE (\ref{ode0}).

\subsection{Proof of Theorem \ref{CRRA2}}\label{proofCRRA2}
We now consider a fixed continuous state-independent strategy $\hat{\bu}$: 
$$\hat{\bu}(t,x)=\hat{\bu}(t),\quad (t,x)\in[0,T]\times(0,+\infty).$$ Following the same procedure as in proving Lemma \ref{lemma31}, we know that the function $f$ defined by 
\begin{equation*}
	f(t,x,z)=\mathbb{E}\left[F\left(\frac{X^{\hat{\bu}}(T)}{z}\right)\bigg|X^{\hat{\bu}}(t)=x\right],\quad (t,x,z)\in[0,T]\times(0,+\infty)\times(0,+\infty),
\end{equation*}
satisfies Assumption \ref{a3} and (\ref{diff}).
Moreover, in a similar way to \eqref{ph}, we  have 
\begin{equation}\label{ph2}
	J(t,x;\hat{\bu})=xe^{B(t)}H(A(t)),\,(t,x)\in[0,T]\times(0,+\infty),
\end{equation}
where
\begin{equation*}
	\begin{split}
		A(t)=&\int_{t}^{T}||\sigma^{\top}(s)\hat{\bu}(s)||^{2}ds,\\ B(t)=&\int_{t}^{T}\Big(\big(1-\textbf{1}\cdot\hat{\bu}(r)\big)_+r(s)
		-\big(\textbf{1}\cdot\hat{\bu}(s)-1\big)_+R(s)+\hat{\bu}^\top(s)\mu(s)\Big)ds-\frac{1}{2}A(t).
	\end{split}
\end{equation*}
{ By a similar way to prove Lemma \ref{lma:GA}, we  obtain from  (\ref{ph2}) and (\ref{diff}) that 
\begin{equation}\label{eq:GA:rR}
    \begin{split}
        -\frac{f_{x}(t,x,J(t,x;\hat{\bu}))}{xf_{xx}(t,x,J(t,x;\hat{\bu}))}=G(A(t)),\quad (t,x)\in[0,T)\times(0,+\infty).
    \end{split}
\end{equation}}

We are now ready to present the proof of the Theorem \ref{CRRA2}.
\begin{proof}[Proof of Theorem \ref{CRRA2}]
Let
\begin{equation*}
\begin{split}
h(t,x,u)=&xf_{x}(t,x,J(t,x;\hat{\bu}))\Big(\big(1-\textbf{1}\cdot u\big)_+r(t)
-\big(\textbf{1}\cdot u-1\big)_+R(t)+u^\top\mu(t)\Big)\\
&  \quad +\frac{1}{2}\lVert\sigma^{\top}(t)u\rVert^{2}x^{2}f_{xx}(t,x,J(t,x;\hat{\bu})),\quad (t,x)\in[0,T]\times(0,+\infty), u\in\mathbb{U}.
\end{split}
\end{equation*}
By Theorem \ref{sn}, $\hat{\bu}$ is an equilibrium strategy if and only if
	\begin{align}
		 &\hat{\bu}(t)=\arg\max\limits_{u\in\mathbb{R}^{d}}\Big\{f_{t}(t,x,J(t,x;\hat{\bu}))+h(t,x,u)\Big\}, \quad (t,x)\in[0,T)\times(0,+\infty),	\label{2pi}\\
		\nonumber	&f_{t}(t,x,J(t,x;\hat{\bu}))+h(t,x,\hat{\bu}(t))=0, \quad (t,x)\in[0,T)\times(0,+\infty)	\nonumber.
	\end{align}
Similar to the proof of Theorem \ref{CRRA},  it is enough to consider the fixed point equation (\ref{2pi}). 
For any $(t,x)\in[0,T]\times(0,+\infty)$ and $u\in\mathbb{U}$, let
\begin{equation*}
	\begin{split}
		h_1(t,x,u)=&xf_{x}(t,x,J(t,x;\hat{\bu}))\Big(r(t)+u^{\top}(\mu(t)-r(t)\mathbf{1})\Big)+\frac{1}{2}\lVert\sigma^{\top}(t)u\rVert^{2}x^{2}f_{xx}(t,x,J(t,x;\hat{\bu})),\\
		h_2(t,x,u)=&xf_{x}(t,x,J(t,x;\hat{\bu}))\Big(R(t)+u^{\top}(\mu(t)-R(t)\mathbf{1})\Big)+\frac{1}{2}\lVert\sigma^{\top}(t)u\rVert^{2}x^{2}f_{xx}(t,x,J(t,x;\hat{\bu})).
	\end{split}
\end{equation*}
Obviously,
\begin{equation*}
	h(t,x,u)=\begin{cases}
		h_1(t,x,u),\quad \textbf{1}\cdot u\leq1,\\
		h_2(t,x,u),\quad \textbf{1}\cdot u>1.
	\end{cases}
\end{equation*}
To proceed, we first consider $\sup_{u\in\mathbb{R}^d}h_1(t,x,u)$ and $\sup_{u\in\mathbb{R}^d}h_2(t,x,u)$, respectively. Obviously,
\begin{align*}
{h_1(t,x,u)\over-x^2f_{xx}(t,x,J(t,x;\hat{\bu}))}
=&-\frac{1}{2}\bigg\lVert\sigma^{\top}(t)u+\frac{f_{x}(t,x,J(t,x;\hat{\bu}))}{xf_{xx}(t,x,J(t,x;\hat{\bu}))}\kappa_{1}(t)\bigg\rVert^{2}\\&+\frac{1}{2}\frac{\left(f_{x}(t,x,J(t,x;\hat{\bu}))\right)^{2}}{\big(xf_{xx}(t,x,J(t,x;\hat{\bu}))\big)^2}\lVert\kappa_{1}(t)\rVert^{2}-{f_{x}(t,x,J(t,x;\hat{\bu}))\over x f_{xx}(t,x,J(t,x;\hat{\bu}))}r(t),\\=&-\frac{1}{2}\bigg\lVert\sigma^{\top}(t)u-G(A(t))\kappa_{1}(t)\bigg\rVert^{2}+\frac{1}{2}(G(A(t))^2\lVert\kappa_{1}(t)\rVert^{2}-G(A(t))r(t).
\end{align*}	
Similarly,
\begin{align*}
{h_2(t,x,u)\over-x^2f_{xx}(t,x,J(t,x;\hat{\bu}))}
=-\frac{1}{2}\bigg\lVert\sigma^{\top}(t)u-G(A(t))\kappa_{2}(t)\bigg\rVert^{2}+\frac{1}{2}(G(A(t))^2\lVert\kappa_{2}(t)\rVert^{2}-G(A(t))R(t).
\end{align*}	
As $f_{xx}(t,x,J(t,x;\hat{\bu}))<0$, we obtain  
\begin{equation*}
	\begin{split}
		&\hat{\bu}_{1}(t):=\arg\max\limits_{u\in\mathbb{R}^d}h_1(t,x,u)=G(A(t))\left(\sigma^{\top}(t)\right)^{-1}\kappa_{1}(t),\\
		&\hat{\bu}_{2}(t):=\arg\max\limits_{u\in\mathbb{R}^d}h_2(t,x,u)=G(A(t))\left(\sigma^{\top}(t)\right)^{-1}\kappa_{2}(t).
		\end{split}
\end{equation*}

Now we consider $\sup_{u\in\mathbb{R}^d}h(t,x,u)$.
Clearly, $\textbf{1}\cdot\hat{\bu}_{2}(t)\leq\textbf{1}\cdot\hat{\bu}_{1}(t)$ for $t\in[0,T]$. Moreover,  $h_{1}(t,x,u)=h_{2}(t,x,u)=h_0(t,x,u)$ for any $(t,x)\in[0,T]\times(0,+\infty)$ if $\textbf{1}\cdot u=1$, where $h_0(t,x,u)$ is defined by
\begin{equation*}
h_0(t,x,u)=xf_{x}(t,x,J(t,x;\hat{\bu}))u^{\top}\mu(t)+\frac{1}{2}\lVert\sigma^{\top}(t) u\rVert^{2}x^{2}f_{xx}(t,x,J(t,x;\hat{\bu})).
\end{equation*}
Then, we have the following three separated cases.
\begin{description}
\item[Case 1:]\label{case1} $1\leq\textbf{1}\cdot\hat{\bu}_{2}(t)$. In this case, we first claim that $$u_{1}(t,x):=\arg\max\limits_{\textbf{1}\cdot u\leq1}h(t,x,u)=\arg\max\limits_{\textbf{1}\cdot u\leq1}h_{1}(t,x,u)=\arg\max\limits_{\textbf{1}\cdot u=1}h_1(t,x,u)=\arg\max\limits_{\textbf{1}\cdot u=1}h_2(t,x,u).$$
Indeed, suppose on the contrary that $\textbf{1}\cdot u_{1}(t,x)<1$, then $u_1(t,x)$ locally and hence globally maximizes $h_{1}(t,x,\cdot)$ by the strict concavity of $h_{1}(t,x,\cdot)$.
Consequently, $u_{1}(t,x)=\hat{\bu}_{1}(t)$, which is impossible as $\textbf{1}\cdot\hat{\bu}_{1}(t)\geq\textbf{1}\cdot\hat{\bu}_{2}(t)\geq1$. 
Then we have 
$\hat{\bu}(t)=\arg\max\limits_{\textbf{1}\cdot u\geq1}h_2(t,x,u)=\hat{\bu}_{2}(t)=G(A(t))\left(\sigma^{\top}(t)\right)^{-1}\kappa_{2}(t)$.

\item[Case 2:]$\textbf{1}\cdot\hat{\bu}_{1}(t)\leq1$. In this case, a similar argument as in {Case 1} yields
$$u_{2}(t,x):=\arg\max\limits_{\textbf{1}\cdot u\geq1}h(t,x,u)=\arg\max\limits_{\textbf{1}\cdot u\geq1}h_{2}(t,x,u)=\arg\max\limits_{\textbf{1}\cdot u=1}h_2(t,x,u)=\arg\max\limits_{\textbf{1}\cdot u=1}h_1(t,x,u).$$

Hence, we have $\hat{\bu}(t)=\arg\max\limits_{\textbf{1}\cdot u\leq1}h_1(t,x,u)=\hat{\bu}_{1}(t)=G(A(t))\left(\sigma^{\top}(t)\right)^{-1}\kappa_{1}(t)$.

\item[Case 3:] $\textbf{1}\cdot\hat{\bu}_{2}(t)< 1< \textbf{1}\cdot\hat{\bu}_{1}(t)$. 
In this case, suppose $\textbf{1}\cdot\hat{\bu}(t)<1$. Then
$\hat{\bu}(t)$ locally 
and hence
globally maximizes $h_1(t,x,\cdot)$, which implies $\hat{\bu}(t)=\hat{\bu}_1(t)$,
contradicting $1< \textbf{1}\cdot\hat{\bu}_{1}(t)$. Therefore,
$\textbf{1}\cdot\hat{\bu}(t)\ge 1$. Similarly, 
$\textbf{1}\cdot\hat{\bu}(t)\le 1$. Consequently, $\textbf{1}\cdot\hat{\bu}(t)=1$ and hence
$$\hat{\bu}(t)=\arg\max\limits_{\textbf{1}\cdot u=1}h(t,x,u)=\arg\max\limits_{\textbf{1}\cdot u=1}h_{0}(t,x,u).$$
Applying the Lagrange multiplier method to $\max\limits_{\textbf{1}\cdot u=1}h_0(t,x,u)$ yields
\begin{equation*}
\hat{\bu}(t)=\left(\sigma(t)\sigma^{\top}(t)\right)^{-1}\bigg[G(A(t))\left(\mu(t)-\frac{\textbf{1}^{\top}\left(\sigma(t)\sigma^{\top}(t)\right)^{-1}\mu(t)}{\textbf{1}^{\top}\left(\sigma(t)\sigma^{\top}(t)\right)^{-1}\textbf{1}}\textbf{1}\right)+\frac{\textbf{1}}{\textbf{1}^{\top}\left(\sigma(t)\sigma^{\top}(t)\right)^{-1}\textbf{1}}\bigg].
\end{equation*}
\end{description}
Hence, $A$ satisfies   ODE (\ref{ode.0}), which admits a unique solution by assumption. Therefore, $\hat{\bu}$ defined by (\ref{hatpi2})	is the unique continuous state-independent equilibrium strategy.
\end{proof}

\subsection{Proof of Lemma \ref{well.}}\label{BB}
	Note that ODE (\ref{ode.0}) is equivalent to the following ODE
\begin{equation}\label{ode.3}
	\begin{cases}
		A'(t)=\mathcal{H}(T-t,A(t)),\quad t\in(0,T],\\
		A(0)=0.
	\end{cases}
\end{equation}
Notice that the function $\mathcal{H}(t,y)$ is defined in a piecewise manner. We first verify its continuity by proving the continuous pasting property at the switching boundaries. For brevity, let $\Sigma(t) = (\sigma(t)\sigma^{\top}(t))^{-1}$. Consider the first switching boundary $G(y)\boldsymbol{1}^{\top}\Sigma(t)(\mu(t) - R(t)\boldsymbol{1}) = 1$, which is equivalent to
\begin{equation*}
    R(t) = \frac{\boldsymbol{1}^{\top}\Sigma(t)\mu(t) - 1/G(y)}{\boldsymbol{1}^{\top}\Sigma(t)\boldsymbol{1}}.
\end{equation*}
Substituting this expression into the  vector $G(y)(\mu(t) - R(t)\boldsymbol{1})$ associated with the first regime, we obtain
\begin{align*}
    G(y)(\mu(t) - R(t)\boldsymbol{1}) 
    &= G(y)\mu(t) - G(y)\frac{\boldsymbol{1}^{\top}\Sigma(t)\mu(t) - 1/G(y)}{\boldsymbol{1}^{\top}\Sigma(t)\boldsymbol{1}} \boldsymbol{1} \\
    &= G(y)\left(\mu(t) - \frac{\boldsymbol{1}^{\top}\Sigma(t)\mu(t)}{\boldsymbol{1}^{\top}\Sigma(t)\boldsymbol{1}} \boldsymbol{1}\right) + \frac{1}{\boldsymbol{1}^{\top}\Sigma(t)\boldsymbol{1}}\boldsymbol{1}.
\end{align*}
This coincides exactly with the corresponding vector in the second regime. Thus, it is clear that the continuity of $\mathcal{H}(t,\cdot)$ across this boundary is established. The continuous pasting at the other switching boundary follows by an analogous argument.

Next, we establish the local Lipschitz continuity of $\mathcal{H}(t,\cdot)$. Fix $t \in [0, T]$. Since $G(\cdot)$ is locally Lipschitz, by the assumptions on the model coefficients, each sub-function of $\mathcal{H}(t, \cdot)$ is locally Lipschitz uniformly in $t$.  Given any compact set $S \subset [0, +\infty)$, let $L^S_1, L^S_2$, and $L^S_3$ be the Lipschitz constants of the three sub-functions with respect to $y$. Denote $L = \max\{L^S_1, L^S_2, L^S_3\} < \infty$. For any $y_1 < y_2$ in $S$, if $y_1$ and $y_2$ belong to the same regime, the result holds trivially. If $y_1$ and $y_2$ cross the internal boundaries, suppose there are $k$ switching points $y_{c,1} < \dots < y_{c,k}$ between them. For unified notation, we set $y_{c,0} = y_1$ and $y_{c,k+1} = y_2$. By the continuous pasting property proved above and the triangle inequality, we obtain
\begin{align*}
	|\mathcal{H}(t, y_1) - \mathcal{H}(t, y_2)| 
	&\le \sum_{i=0}^{k} |\mathcal{H}(t, y_{c,i}) - \mathcal{H}(t, y_{c,i+1})| \\
	&\le L \sum_{i=0}^{k} (y_{c,i+1} - y_{c,i}) \\
	&= L (y_2 - y_1) = L|y_1 - y_2|.
\end{align*}
This shows that $\mathcal{H}(T-t,A)$ remains locally Lipschitz with respect to $A\in[0,+\infty)$, and this property is uniform in $t\in[0,T]$. Therefore, the local wellposedness of  ODE (\ref{ode.3}) is obtained.
 
 Moreover, because, for any $t\in[0,T]$
 \begin{equation*}
 	|\mathcal{H}(T-t,y)|\leq C\left(1+|G(y)|^{2}\right)\leq M:= C\left(1+\sup\limits_{y\in[0,+\infty)}|G(y)|^{2}\right),
 \end{equation*}
where $C$ is some constant, it follows that the local solution does not blow up on $[0,T]$, thus the proof follows.

\end{appendices}

\bibliographystyle{siam}
\bibliography{ref}

\end{document}